\documentclass[12pt]{amsart}
\usepackage{amssymb}
\usepackage[latin1]{inputenc}
\usepackage{amsthm}
\usepackage{amsfonts}
\usepackage{mathrsfs}

\usepackage{graphicx}
\usepackage{amsmath}
\usepackage[all]{xy}

\let\mathcal\mathscr

\def\C{{\mathbb C}}

\def\bP{{\mathbb P}}
\def\bN{{\mathbb N}}
\def\Q{{\mathbb Q}}
\def\Z{{\mathbb Z}}
\newtheorem{thm}{Theorem}[section]

\def\Z{{\bf Z}}
 
\def\P{{\bf P}}

\def\R{{\bf R}}
\def\Q{{\bf Q}}
\def\C{{\bf C}}

\def\Spec{\mathop{\rm Spec}\nolimits}

\def\tilde{\widetilde}

\def\phi{\varphi}
\def\a{{\alpha}}

\def\cB{{\mathcal B}}
\def\cC{{\mathcal C}}
\def\cD{{\mathcal D}}

\def\cH{{\mathcal H}}

\def\cL{{\mathcal L}}
\def\cM{{\mathcal M}}
\def\cO{{\mathcal O}}

\def\cU{{\mathcal U}}
\def\cV{{\mathcal V}}

\def\cX{{\mathcal X}}
\def\cY{{\mathcal Y}}

\def\div{\mathop{\rm div}\nolimits}

\def\Card{\mathop{\rm Card}\nolimits}

\numberwithin{equation}{section}

\newtheorem{theorem}[thm]{Theorem}

\newtheorem{conjecture}[thm]{Conjecture}
\newtheorem{corollary}[thm]{Corollary}

\newtheorem{definition}[thm]{Definition}

\newtheorem{lemma}[thm]{Lemma}

\newtheorem{proposition}[thm]{Proposition}

\newtheorem{remark}[thm]{Remark}

\newtheorem{Fact}[thm]{Fact}

\begin{document}

\title {Transcendental Liouville inequalities on projective varieties}
\author{Carlo Gasbarri}\address{Carlo Gasbarri, IRMA, UMR 7501
 7 rue Ren\'e-Descartes
 67084 Strasbourg Cedex }
 \email{gasbarri@math.unistra.fr}

\date\today
\thanks{Research supported by the project ANR-16-CE40-0008 FOLIAGE}


\begin{abstract} Let $p$ be an algebraic point of a projective variety $X$ defined over a number field. Liouville inequality tells us that the norm  at $p$ of a non vanishing integral  global section of an hermitian  line bundle over $X$ is either zero or it cannot be too small with respect to the $\sup$ norm of the section itself. 
We study inequalities similar to Liouville's for subvarietes and for transcendental points of a projective variety defined over a number field. We prove that almost all transcendental points verify a good inequality of Liouville type. We also relate our methods to a (former) conjecture by Chudnovsky and give two applications to the growth of  the number of rational points of bounded height on the image of an analytic map from   a disk to a projective variety.
\end{abstract}

\maketitle

\tableofcontents
\section {Introduction}

An important tool in diophantine geometry and in transcendental theory is the so called {\it Liouville inequality}. In its simplest form, it may be stated in the following way: Let $n\in\Z$ and $P(z)\in\Z[z]$ be a polynomial, then $P(n)\neq 0$ implies that $|P(n)|\geq 1$. This can be generalized to any algebraic number and also to any algebraic point of a projective variety: Suppose that $\cX\to\Spec(\Z)$ is a projective arithmetic scheme and  $\cL$ is an hermitian ample line bundle on it.  In this introduction, in order to simplify notations, we suppose that everything is defined over $\Z$ (the case of varieties defined over arbitrary number fields will be treated in the paper). If $s\in H^0(\cX, \cL^d)$, we denote by $\Vert s\Vert$  the supremum of $\Vert s\Vert(z)$ on $\cX(\C)$. 

Liouville inequality tells us that, if $x\in\cX(\overline{\Q})$ is an algebraic point, we can find a positive constant $A$ such that, for every positive integer $d$ and every $s\in H^0(\cX,\cL^d)$ such that $s(x)\neq 0$ we have $\log\Vert s\Vert (x)\geq -A(\log\Vert s\Vert +d)$ cf. Theorem \ref{Liouvillethm}. 

In section 3 we show that a similar inequality holds for sub varieties: Denote by $X$ the generic fibre of the projective arithmetic scheme $\cX$ fixed above. For every closed sub variety  $Y$ of $X(\C)$ and every global section $s\in H^0(\cX, \cL^d)$, we denote by $\Vert s\Vert_Y$ the supremum of $\Vert s\Vert$ on $Y$. In section 3 we prove, 
\begin{theorem} (cf. Theorem \ref{liouville2}) Let $Y$ be a  Zariski closed sub variety of $X(\C)$ which is defined over $\overline{\Q}$, then we can find a constant $A$ , depending only on $Y$ and the couple $(\cX,\cL)$, such that, for every positive integer $d$ and every $s\in H^0(\cX, \cL^d)$ which do not vanish identically on $Y$ we have
$$\log\Vert s\Vert_Y\geq -A(\log\Vert s\Vert +d).$$
\end{theorem}

Usually, In many diophantine proofs, Liouville inequality is used,  together with a  Siegel lemma, some form of Schwartz inequalities and  Zero Lemmas. The Zero Lemmas are used to ensure that the section we are dealing with do not vanish on involved algebraic point. Siegel lemma and Schwartz inequalities provide upper bounds for the value of the norm of a section in the point and Liouville inequality provides a lower bound of the same.  Even if we can find in the literature many effective versions of Zero Lemmas, when one applies them, usually the price to pay is the effectivity of the statement:  In order to apply the Zero Lemmas often it is required that the points we are dealing with have some special properties which can be supposed to hold just because the strategy of proof is by contradiction. for instance, in the actual proofs, the effectivity  of Roth theorem is lost in this way. If one could replace some of the involved algebraic points by transcendental points, probably we could avoid the the use of Zero Lemmas  and we could gain in effectivity. Unfortunately this cannot be done directly because, when one deals with transcendental points, Liouville inequality is not available anymore. Never the less,  the last section of this paper, we will see some applications of these ideas. 

In this paper we will study some weak form of Liouville inequality which holds for "many" points of an algebraic variety. 

In section 5 we  study  inequalities similar to Liouville inequality which are  in the spirit of the inequalities proposed by Chudnovsky, cf. \cite{CHUDNOVSKY}. 

If $P(z)\in \Z[z_1,\dots , z_N]$ is a polynomial, we will denote by $\Vert P\Vert$ the maximum of the absolute values of its coefficients. 

In the paper \cite{CHUDNOVSKY} the author conjectured that the set of points $\zeta\in \C^N$ for which there is a constant $A$ (depending on $\zeta$) such that, for every $P(z)$ of degree $d$, we have $\log|P(\zeta)|\geq -A(\log\Vert P\Vert +d)^{N+1}$ is full in $\C^N$. This conjecture have been proved by Amoroso \cite{AMOROSO}.
 
Here we prove an inequality which is independent but may be stronger in some situations:

\begin{theorem}\label{choudno intro} (cf. Definition \ref{S points on CN} and Theorem \ref{amoroso1}) Suppose that $a\geq N+1$.  A point $\zeta\in \C^N$ is said to be of type $S^N_a(\C)$ if we can find  a positive constant $A$ depending on $\zeta$, such that, for every non zero polynomial $P(z)\in \Z[z_1,\dots , z_N]$ of degree $d$ we have
\begin{equation}\label{eq: SNforpoly}
\log\vert P(\zeta)\vert\geq -Ad(d^{a-1}\log\Vert P\Vert+d\log(d)+1).
\end{equation} Then the set of points of type $S^N_a(\C)$ is full in $\C^N$ for the standard Lebesgue measure. 
\end{theorem}

We would like to observe that the proof of the theorem above do not use special properties of the field $\C$. For instance it holds over $p$--adic fields. In section \ref{padicfields} we give some details of the proof in the $p$--adic case.  Remark that the proof by Amoroso in \cite{AMOROSO} holds only over $\C$. 

\smallskip

In the following sections will study some possible inequalities of Liouville type which hold for almost all the transcendental points on a projective variety.  

We will say that a point $z\in\cX(\C)$ is {\it arithmetically generic} if, {\it for every positive integer $d$} and every non zero section $s\in H^0(\cX, \cL^d)$ do not vanish at $z$. In section \ref{arithmetically generic sub varieties} we verify that arithmetic generic points do not depend on the line bundle $\cL$ but only on the generic fibre of $\cX$ and the complex embedding. 

First of all we show that inequalities as good as Liouville's are not possible for transcendental points ); We suppose that the genenic fiber of $\cX$ is of dimension $N$ and that $z\in\cX(\C)$:

\begin{theorem} (cf. Theorem \ref{liouville3}) We can find a constant $A$ depending only on $(\cX,\cL)$ for which the following holds:  let $z\in X_{\sigma_0}(\C)$ be an arithmetically generic point, then, for every $d\in \bN$ there exists an infinite sequence of sections $s_n\in H^0(\cX,\cL^d)$ for which
\begin{equation} \log\Vert s_n\Vert(z)\leq -Ad^N(\log^+\Vert s_n\Vert +d).
\end{equation}
\end{theorem}

(We define $\log^+(\cdot):=\sup\{ \log(\cdot),0\}$),

For this reason, in this paper we introduce the notion of $S_a(\cX)$ points, these are points for which a weaker inequality in the other direction holds.

\begin{definition}\label{SNpoints intro} (cf. Definition \ref{genericsubvariety}) Let $a\geq N$ be a real number. We will say that $z$ is of type $S_a(\cX)$  if we can find  a positive constants $A=A(z,\cL, a)$, depending on $z$, $\cL$  and $a$ such that, {\rm for every positive integer $d$ and every non zero global section $s\in H^0(\cX, \cL^d)$}, we have that
\begin{equation}\label{liouville inequality intro}
\log\Vert s\Vert(z)\geq -Ad^a(\log^+\Vert s\Vert +d).
\end{equation}
We will denote by $S_a(X)$ the subset of $X(\C)$ of points of type $S_a(\cX)$. \end{definition}

One can prove that the set $S_a(X)$ is independent on the choice of the hermitian ample line bundle $\cL$.



The hermitian metric on the ample line bundle $\cL$ induces a metric  $\mu_\cL(\cdot)$ on $\cX(\C)$.  Given two hermitian line bundles $\cL_1$ and $\cL_2$ on $\cX$, the induced metrics  $\mu_{\cL_i}(\cdot)$ may not coincide but, since $\cX(\C)$ is compact,  a subset $U\subset \cX(\C)$ is full for $\mu_{\cL_1}(\cdot)$ (this means that $\mu_{\cL_1}(U)=\mu_{\cL_2}(\cX(\C))$) if and only if it is full for $\mu_{\cL_2}(\cdot)$. Thus we can speak about full subsets of $\cX(\C)$ without making reference to the measure on it. 

We will then prove:

\begin{theorem}\label{full Sa on varieties intro} (cf. Theorem \ref{SNFULL}) For every $a\geq N+1$, the set of points of type $S_a(\cX)$ is full for the Lebesgue measure on $\cX(\C)$ (the complementary is of zero Lebesgue measure). 
\end{theorem}
Usually theorem  of this type are proved by coupling the classical Borel--Cantelli lemma (cf. \ref{borelcantelli}) with an estimate of the volume of the set of elements having norm which do not satisfy the inequality of definition \ref{SNpoints intro}. The computation of this volume is usually quite delicate and the standard strategy is to relate it to the distance of the points to the zero set of the involved global section. This is the strategy used for instance by Lang in \cite{LANG}, by  Amoroso in \cite{AMOROSO} and in many proofs of the book \cite{BUGEAUDBOOK}. Here we use a strategy which is a bit different: by Fubini Theorem, we can reduce the proof to a similar computation on compact Riemann surfaces. Over one dimensional disks we use an argument involving Lagrange interpolation which allows to estimate, given a {\it polynomial} $P(z)$, the area of the the set of points $z$ for which $|P(z)|\leq\epsilon$. Then we approximate analytic functions by polynomials. In order to obtain good approximation, an important criterion of algebraicity due to Sadullaev \cite{SADULLAEV} is used.  The proof of the Theorem of Sadullaev makes a heavy use of peculiar properties of complex analysis, which prevents a straightforward generalization of Theorem \ref{full Sa on varieties intro} to a $p$--adic contest.  We think that a $p$--adic version of \cite{SADULLAEV} would be very important and interesting.

One should observe that inequalities of Theorem \ref{choudno intro} seems better then inequalities of Theorem  \ref{full Sa on varieties intro}. Never the less they are obtained essentially by the same methods. This due to the fact that the exponent of the degree in the involved inequalities is related to the arithmetic ampleness of the involved line bundle. When we work with polynomials with standard height, the involved line bundle is a arithmetically nef but not an arithmetically  big  line bundle. The fact that we can obtain better inequalities in this contest reflects this difference. 

\smallskip

In the last section we describe two application to the growth of rational points of bounded height in the image of an analytic map of a disk to a projective variety. These are in the spirit of the principle explained above: the presence of a "good" transcendental point can be a tool which replaces Zero Lemmas.

Suppose that $X$ is a smooth projective variety of dimension $N>1$ defined over $\Q$ (as before, a statement for general number field holds but, in order to simplify notations, in this introduction we state it only over $\Q$). As before we suppose that we fixed a model $\cX$ of it  and an hermitian ample line bundle $\cL$ over $\cX$. 

Let $\Delta_1$ be the unit disk in $\C$ and $f:\Delta_1\to X(\C)$ be an analytic map with Zariski dense image. Fix $r<1$ positive. We are interested in studying behavior of the cardinality $C_r(f,T)$ of the set
\begin{equation}
S_r(f,T):=\{ z\in \Delta_1\; /\; |z|<r\; ; \; f(z)\in X(\Q)\; {\rm and}\; h_\cL(f(z))\leq T\}
\end{equation}
when $T$ goes to infinity. 

There is a huge literature on this problem: the classical theorem by Bombieri and Pila \cite{BOMBIERIPILA} tells us that in general $C_r(f,T)\ll \exp(\epsilon T)$ but there are many interesting and natural conditions, which, if satisfied, imply that the growth of $C_f(r,T)$ is polynomial in $T$, cf. for instance \cite{COMTEYOMDIN} or \cite{BOXALL} and \cite{MASSER}.

The theorems we can prove are:

\begin{theorem}\label{counting0} (cf. Theorem \ref{counting2}) Let $a\geq N+1$. Suppose that there is $z_0\in\Delta$ such that $f(z_0)\in S_a(X)$.  Then, for every $\epsilon>0$ and $\gamma>{{1}\over{N}}$ there exists a constant $A=A(\cX,\cL, r,f,\epsilon, \gamma, a)$ such that, if $T\geq A$ then we have
\begin{equation}
C_f(r,T)\leq \epsilon T^{\gamma(a+1)+1}.
\end{equation}
\end{theorem}

and

\begin{theorem}\label{counting01} (cf. Theorem \ref{counting3})  Let $f:\Delta_1\to X(\C)$ as before. Let $a\geq N-1$ be a real number. Let $s_0\in H^0(\cL, \cL)$ be an irreducible smooth divisor. Suppose that  there is $p\in f(\Delta_r)\cap \div(s_0)$ which is of type $S_a(\div(s_0))$. Then, for every $\epsilon>0$ and $\gamma>{{1}\over{N}}$ there exists a constant $A=A(\cX,\cL, r,f,\epsilon, \gamma)$ such that, if $T\geq A$ then we have
\begin{equation}
C_f(r,T)\leq \epsilon T^{\gamma(a+1)+1}.
\end{equation}
\end{theorem}

These theorems  should also be compared with a similar theorem of proved by Surroca \cite{surroca}. In her theorem there is no condition on the type of points contained in the image but the conclusion is that the polynomial growth holds only for a sequence of values of $T$. Her paper provides also counterexamples which show that her theorem is sharp. consequently, her counterexamples provide examples of maps  which do not verify the hypothesis of Theorems \ref{counting0} and \ref{counting01}. The reader should also compare with the paper \cite{BOXALL} where similar ideas are exploited. 

The proof of Theorems \ref{counting0} and \ref{counting01} exploits the generalized Liouville inequality for points of type $S_a(\cX)$ plus a Siegel Lemma which bounds the degree $d$ and the height of a section of $\cL^d$ vanishing on $S_r(f,T)$: cf. \ref{vanishingsmallsections}.

Few words about the pre existing literature: Chudnovsky in \cite{CHUDNOVSKY} stated his conjecture in relation with possible criteria of algebraic independence between complex numbers. Amoroso proved the conjecture in \cite{AMOROSO} over the complex numbers and another proof which holds also over $p$--adic fields is given by Mikha$\breve{\rm i}$lov in \cite{MIKHAILOV}. Both proofs rely on the commutative algebra tools developed by Nesterenko. Even if very similar, one cannot deduce our Theorem \ref{choudno intro} from theirs and conversely. Moreover our proof is much shorter and elementary. 

At the moment, to our knowledge, there is no literature on possible inequalities of Liouville type on general projective varieties. 

In the seminal papers \cite{BKM} and \cite{KM}, the authors study a similar problem for arbitrary sub varieties of the affine spaces. They obtain very deep results on the fullness of "badly" approximable by linear forms points on an arbitrary variety. In \cite{BKM} this is done in the spirit of Khitchine theory.   

Translated in the language used in this paper, their theory would imply the following result: Let $(\cX, \cL)$ be as before {\it and fix a positive integer $d$}. The set of points $z\in\cX(\C)$ for which there is a constant $A$ (depending on $z$, $\cX$, $\cL$, $\epsilon$ {\it and $d$}) such that, for every $s\in H^0(\cX;\cL^d)$ we have $\log\Vert s\Vert(z)\geq -A(\log^+\Vert s\Vert +d)^{N+\epsilon}$, is full in $\cX(\C)$. 

Even if the exponent in the inequality is better then what we find, the involved constant $A$ depends, a priori, on $d$. Thus, even when coupled with a Baire Theorem argument, their theorems do not seems to imply ours. Indeed, the constant involved in the inequalities we study are {\it independent on $d$}.

\section{Notations and basic facts from Arakelov geometry}\label{notations etc}

\subsection{Tools and notations from arithmetic geometry and Arakelov theory} 

Let $K$ be a number field and $O_K$ be its ring of integers. We will denote by $M_K^\infty$ the set of infinite places of $K$. We fix a place $\sigma_0\in M_K^\infty$. 

Let $X_K$ be a projective variety of dimension $N$ defined over $K$

If $\tau\in M_K^\infty$  and $F$ is an object over $X_K$ ($F$ may be a sheaf, a divisor, a cycle...), we will denote by $X_\tau$ the complex variety $X_K\otimes_\tau\C$ and by $F_\tau$ the restriction of $F$ to $X_\tau$. 

A model $\cX\to\Spec(O_K)$ of $X_K$ is a  flat projective $O_K$ scheme whose generic fiber is isomorphic to $X_K$. Suppose that $L_K$  and $\cX$ are  respectively a line bundle over $X_K$ and a model of it; We will say that a line bundle $\cL$ over $\cX$ is a model of $L_K$ if its restriction to the generic fiber is isomorphic to $L_K$. 

If $\cX$ is a model of $X_K$, an hermitian line bundle $\overline\cL=(\cL, \Vert\cdot\Vert_\sigma)_{\sigma\in M_K^\infty}$ is a line bundle over it equipped, or every $\tau\in M_K^\infty$ a metric on $L_\tau$ with the condition that, if $\sigma=\overline\tau$ then the metric over  $L_\sigma$ is the conjugate of the metric on $L_\tau$.

If $X_K$ is a projective variety, it is easy to see that for every line  bundle $L_K$  on $X_K$, we can find an embedding $\iota : X_K\hookrightarrow P_K$,  where $P_K$ is a smooth projective variety and $L=\iota^\ast(M)$ with $M$ line bundle on $P_K$. A metric on $K_L$ will be said to be smooth if it is the restriction of a smooth metric on $M$.

Let $\overline\cL=(\cL, \Vert\cdot\Vert_\sigma)_{\sigma\in M_K^\infty}$ be an hermitian line bundle on a model $\cX$  of $X_K$. If $s\in H^0(X_K, L^d_K)$ is a non zero section, we will denote by $\log^+\Vert s\Vert$ the real number $\sup_{\tau\in M_K^\infty}\{ 0,\log\Vert s_\tau\Vert_\tau\}$. More generally, of $a$ is a real number,  we will denote by $a^+$ the real number $\sup\{a,0\}$ and by $a_+$ the real number $\sup\{ 1, a\}$.

\begin{definition} An {\rm arithmetic polarization $(\cX,\overline\cL)$ of $X_K$} is the choice of the following data:

\itemize
\item An ample line bundle $L_K$ over $X_K$
\item A projective model $\cX\to\Spec(O_K)$ of $X_K$ over $O_K$.
\item A relatively ample line bundle hermitian line bundle $\overline\cL$ over $\cX$ which is a model of $L_K$.
\item For every $\tau\in M_K^\infty$ we suppose that the metric on $L_\tau$ is smooth and positive.
\enditemize
\end{definition}

We recall the following standard facts of Arakelov theory:
\itemize

\item If $L$ is a hermitian line bundle over $\Spec(O_K)$ and $s\in L$ is a non vanishing section, we define
\begin{equation}
\widehat{\deg}(L):=\log(Card(L/sO_k))-\sum_{\sigma_in M^\infty_K}\log\Vert s\Vert_\sigma.
\end{equation}
If $E$ is an hermitian vector bundle of rank $r$ on $\Spec(O_K)$, we define $\widehat{\deg}(E):=\widehat{\deg}(\wedge^{r} E)$ and the slope of $E$ is $\widehat{\mu}(E)={{\widehat{\deg}(E)}\over{r}}$.  

\item Within all the sub bundles of $E$ there is one whose slope is maximal, we denote by $\widehat{\mu}_{\max}(E)$ its slope. It is easy to verify that $\widehat{\mu}_{\max}(E_1\oplus E_2)=\max\{ \widehat{\mu}_{\max}(E_1), \widehat{\mu}_{\max}(E_2)\}$.

\item We will need the following version of the Siegel Lemma: 

\begin{lemma}\label{siegel} (Siegel Lemma) Let $E_1$ and $E_2$ be hermitian vector bundles over $O_K$. Let $f:E_1\to E_2$ be a non injective linear map. Denote by $m=rk(E_1)$ and $n=rk(Ker(f))$. Suppose that there exists a positive real constant $C$ such that:

a) $E_1$ is generated by elements of $\sup$ norm less or equal then $C$.

b) For every infinite place $\sigma$ we have $\Vert f\Vert_\sigma\leq C$ 

Then there exists an non zero element $v\in Ker(f)$ such that
\begin{equation}
\sup_{\sigma\in M^\infty_K}\{ \log\Vert v\Vert_\sigma\}\leq{{m}\over{n}}\log(C^2)+\left( {{m}\over{n}}-1\right) \widehat{\mu}_{\max}(E_2)+3\log(n)+A
\end{equation}
where $A$ is a constant depending only on $K$. 
\end{lemma}

A proof of this version of Siegel Lemma can be found in \cite{gasbarri}.

\item Let $L$ be an hermitian ample line bundle on a projective variety $Z$ equipped with a smooth metric $\omega$. We suppose that the metric on $L$ is smooth. Over $H^0(Z, L^d)$ we can define two natural norms:
 \begin{equation}
 \Vert s\Vert_{\sup}:=\sup_{z\in Z}\Vert s\Vert (z)\}\;\; {\rm and}\;\;\Vert s\Vert_{L^2}:=\sqrt{\int_Z\Vert s\Vert^2 \omega^n}.
 \end{equation}
These norms are comparable: we can find constants $C_i$ such that

\begin{equation}\label{gromov}
C_1\Vert s\Vert_{L^2}\leq \Vert s\Vert_{\sup}\leq C_2^d\Vert s\Vert_{L^2}.
\end{equation}

This statement (due to Gromov) is proved for instance in \cite{SABK} Lemma 2 p. 166 when $Z$ is smooth. The general statement can be deduced by taking a resolution of singularities (remark that the proof of \cite{SABK} Lemma 2 p. 166  do not require that $L$ is ample). 

\item Suppose that $L$ is an hermitian ample line bundle on a smooth projective variety $X$ defined over $\C$. Then we can find a constant $C$ for which the following holds: for every couple of positive integers $d_1$ and $d_2$ and non vanishing global sections $s_1\in H^0(X, L^{d_1})$ and $s_2\in H^0(X,L^{d_2})$ we have
\begin{equation}\label{eq:supestimate}
\log\Vert s_1\Vert_{\sup}+\log\Vert s_2\Vert_{\sup}\geq \log\Vert s_1\cdot s_2\Vert_{\sup}\geq (d_1+d_2)C+\log\Vert s_1\Vert_{\sup}+\log\Vert s_2\Vert_{\sup}.
\end{equation}

\item If $(\cX,\cL)$ is an arithmetic polarization of $X_K$, then we we can find constants  $C_1$ and $C_2$  such that
\begin{equation}\label{eq:counting1}C_1^{d^{N+1}} T^{d^N}\leq \Card\left(\{ s\in H^0(\cX, \cL^d)\;\; /\sup_{\tau\in M_K^{\infty}}\{\Vert s\Vert_\tau\}\leq T\}\right)\leq C_2^{d^{N+1}} T^{d^N}.\end{equation}
This is a consequence of \cite{Zhang}, Theorem 1.4 , \cite{GILLETSOULEISTRAEL} Theorem 2 and the comparaison above. 

\item If $\cL$ is an arithmetically ample line bundle, then for $d$ sufficiently big, the lattice $H^0(\cX,\cL)$ is generated by sections of $\sup$ norm less or equal then one. Cf. \cite{Zhang} for a proof. 

\item Let $L/K$ be a finite extension and $O_L$ the ring of integers of $L$. An $L$--point of $X_K$ is a $K$--morphism $P_L:\Spec(L)\to X_K$. The set of $L$ points of $X_K$ is noted  $X_K(L)$. If $(\cX,\cL)$ is an arithmetic polarization of $X_K$, by the valuative criterion of properness, every $L$--point $P_L:\Spec(L)\to X_K$ extends uniquely to a $O_K$--morphism $P_{O_L}:\Spec(O_L)\to \cX$. In this case, $P_{O_L}^\ast(\cL)$ is an hermitian line bundle on $\Spec(O_L)$. We define the height of $P_L$ with respect to $\cL$ to be the real number $h_\cL(P_L):={{\widehat{\deg}(P_L^\ast(\cL))}\over{[L:\Q]}}$.

\subsection{Tools and notations from measure theory}

Let $X$ be  a smooth projective variety  of dimension $N$ defined over $\C$. A positive $(1,1)$ form $\omega$ on $X$ induces a volume form $\omega^N$ and consequently a measure $\mu_\omega(\cdot)$ on $X$.  Let  $A\subset X$ be a subset. We will say that {\it $A$ is full in $X$} if the measure of $\mu_\omega(X\setminus A)=0$.  If $\omega_1$ is another positive $(1,1)$ form on $X$, then by compactness of $X$ it is easy to see that $\mu_\omega(X\setminus A)=0$ if and only if $\mu_{\omega_1}(X\setminus A)=0$; thus the "fullness" of $A$ is independent on the chosen metric. 

A key tool in this paper is the classical Theorem of Borel--Cantelli, which can be found in any standard book in measure theory:

\begin{proposition} \label{borelcantelli} Let $X$ be a variety equipped with the Lebesgue measure $\mu$. Let $\{A_n\}_{n\in \bN}$ be a sequence of measurable sets of $X$ such that 
$$\sum_{n=1}^\infty\mu(A_n)<\infty$$
then
$$\mu(\bigcap_{n=1}^\infty\bigcup_{k\geq n}A_k)=0.$$
That means that almost all $x\in X$ belong only to finitely many $A_n$. 
\end{proposition}

\section{Liouville lower bound for sub varieties}

We fix a complex embedding $\sigma_0\in M_K$. Suppose that $X_K$ is a projective variety and $(\cX, \cL)$ is an arithmetic polarization over it. The very definition of the height of an algebraic point gives us a lower bound of the value of a global section of $\cL^d$ on it in terms of the its $\sup$ norm:

\begin{theorem}\label{Liouvillethm} Let $p\in X_K(\overline{K})$ be an algebraic point. Let $p_0\in X_{\sigma_0}$ be its image. Then we can find a constant $A$, depending on $X_K$, the point $p$ and the arithmetic polarization, for which the following holds: for every positive integer $d$ and  global section $s\in H^0(\cX, \cL^d)$ such that $s(p)\neq 0$ we have
\begin{equation}\label{liouville1} 
\log\Vert s\Vert_{\sigma_0}(p)\geq-A(\log\Vert s\Vert+d).
\end{equation}
\end{theorem}
\begin{proof} Denote by $G$ the Galois group of $\overline{K}$ over $K$. Let $N=[K(p):K]$ and $\{ p=p_0, p_1,\dots, p_N\}$ the orbit of $p$ under $G$. For every $\sigma\in M_K^\infty$ denote by $p_i^\sigma$ the image of $p_i$ in $X_\sigma$.

Let $s\in H^0(\cX, \cL^d)$ not vanishing at $p$. By definition of height (computed as Arakelov degree), we have that
\begin{equation}
dh_\cL(p)\geq -{{1}\over{[K(p):\Q]}}\cdot\sum_{\sigma\in M_K^\infty}\sum_{i=0}^N\log\Vert s\Vert_\sigma(p_i^\sigma).
\end{equation}
From this we obtain
\begin{equation} \log\Vert s\Vert_{\sigma_0}(p_0)\geq -d[K(p):\Q]\cdot h_\cL(p)-\sum_{i=1}^N\log \Vert s\Vert_{\sigma_0}(p_i^{\sigma_0})-\sum_{\tau\neq \sigma_0}\sum_{i=0}^N\log\Vert s\Vert_{\tau}(p_i^\tau).\end{equation}
The conclusion easily follows. 
\end{proof}

A corollary of the proof of Theorem \ref{Liouvillethm} is:

\begin{corollary}\label{Liouville and heights} Let $p\in X_K(\overline{K})$. The involved constant $A$ in Theorem \ref{Liouvillethm} can be taken to be $[K(p):\Q]\cdot h_\cL(p)$. 
\end{corollary}

Thus, given an algebraic point,  the minimal constant which for which an inequality of Liouville holds, is the height of the point itself. 

We would now like to show that  a similar inequality holds for closed sub varieties of $X_K$: If $Z\subset X_K$ is a closed sub variety and if  $s\in H^0(\cX,\cL^d)$ is a non zero section non vanishing identically on $Z$, we will denote $$\log\Vert s\Vert_{Z,\sigma_0}:=\sup\{\log\Vert s\Vert_{\sigma_0}(z) \;/ \; z\in Z_{\sigma_0}(\C)\}.$$

\begin{theorem}\label{liouville2} Let $Y_K$ be a closed sub variety of $X_K$, then we can find a constant $A$, depending only on $K$,  $Y_K$, $X_K$ and the arithmetic polarization for which the following  holds: For every positive integer $d$ and global section $s\in H^0(\cX,cL^d)$ not vanishing on $Y_K$
\begin{equation} \log\Vert s\Vert_{Y_K,\sigma_0}\geq -A(\log\Vert s\Vert +d).\end{equation}

\end{theorem}

\begin{proof} In order to prove the theorem we need the following lemma:

\begin{lemma}\label{bounded height} Let $Z$ be a projective variety and $L$ an ample line bundle over it. Let $h_L(\cdot)$ be a height associated to $L$. Then there exists a Zariski dense set of points $S\subset Z(\overline{K})$ and a constant $C$, depending only on $Z$,  $K$ and the choice of the height $h_L(\cdot)$ such that, for every $p\in S$ we have
\begin{equation}
h_L(p)\leq C
\end{equation}
\end{lemma}

Recall that the height is normalized  by dividing the Arakelov degree by the degree over $\Q$ of a field containing the coordinate of the points. 

Let's show first how the lemma implies the theorem: Let $s$ be the global section as in the Theorem. By the Lemma, we can find a point $p\in S$ such that $s(p)\neq 0$. Now the proof follows the proof of Theorem \ref{liouville1}: using the notation of Theorem \ref{liouville1}, we have that
\begin{equation}
dh_L(p)\geq -{{1}\over{[K(p):\Q]}}\cdot\sum_{\sigma\in M_K^\infty}\sum_{i=1}^N\log\Vert s\Vert_\sigma(p_i^\sigma).
\end{equation}
Thus
\begin{eqnarray*}
{{1}\over{[K:\Q]}}\cdot\log\Vert s\Vert_{Y,\sigma_0}\geq &-dh_L(P)-{{[K:\Q]-1}\over{[K:\Q]}}\cdot\log\Vert s\Vert\\
\geq &-Cd-{{[K:\Q]-1}\over{[K:\Q]}}\cdot\log\Vert s\Vert.
\end{eqnarray*}
The conclusion follows. \end{proof}

\begin{proof} {(\it of Lemma  \ref{bounded height})} we first deal with the case when $X_K=\P^N$, $L=\cO(1)$ and the height is the standard Weil Height. In this case, it suffices to observe that the set of points having homogeneous coordinates which are roots of unity will satisfy the conclusion of the theorem. Denote by $S_\P$ this set. 

In  order to attack the general case, it suffices to observe that the set $S$ (if it exists) is independent on the choice of the ample line bundle and the representative of the height.  Thus, let $p:X_K\to \P^{\dim(X_K)}$ be a finite morphism and $L:=p^\ast(\cO(1))$. The line bundle $L$  is ample and $S:=p^{-1}(S_\P)$ will satisfy the conclusion of the Lemma.  \end{proof} 

Theorems \ref{liouville1} and \ref{liouville2} tell us that the algebraic sub varieties defined over the algebraic closure of $K$ satisfy what we would call a {\it Liouville property}: non vanishing  integral sections cannot be too small over them. This is evidently similar to the classical Liouville property of algebraic numbers which is one of the main tools in diophantine approximation and transcendence theory. On the other side, one can find sub varieties which are not defined over $\overline{K}$ and for which a similar property do not hold. For instance there exist points $\alpha$ of $\bP^1(\C)$  for which we can find a sequence of polynomials $P_n(z)\in \Z[z]$ and a diverging to infinity sequence $\omega_n$ such that $\log\Vert P_n(\alpha)\Vert \leq -\omega_n\log\Vert P_n\Vert$.

\begin{remark} One can deduce from the results of \cite{Zhang} that, in this case too, the minimal constant for which Theorem \ref{liouville2} holds is the height of $Y_K$ with respect to  $\cL$. Since we will not need this, we will not prove it here.\end{remark}

In the following sections we will study what kind of inequalities of "Liouville type" we can hope on transcendental points. 

\section{An area computation on disks}

In this section we will prove a technical  Lemma which will be used in the paper. We want to compute the area of  the points of a disk where an analytic function is "small".

In the sequel, we will denote by $\Delta_r$ the disk $\{ z\in\C \;/\; |z|<r\}$. An hermitian line bundle $L$ on $\Delta_r$ is just the trivial line bundle $\cO_{\Delta_r}$ equipped with a non vanishing smooth positive function $\rho(z):=\Vert 1\Vert(z)$.  More precisely, a line bundle on $\Delta_r$ is of the form $\cO_{\Delta_r}\cdot e$, where $e(z)$ is a non vanishing global section of $L$. If $f\in H^0(\Delta_r,L^d)$ then $\Vert f\Vert(z) =|f|(z)\cdot\rho^d(z)$. For every $r_0<r$ and section $f\in H^0(\Delta_r; L)$, we will denote $\Vert f\Vert_{r_0}:=\sup\{ \Vert f\Vert (z)\;/\; z\in \Delta_{r_0}\}$ (if $\rho=1$ then we write $\Vert f\Vert_r=\vert f\vert_r$). 

Let $C>1$ be a fixed real number. and $r_0<r_1<r$.  For every positive integer $d$, let 
\begin{equation}\label{definition of B} 
B_C(r_0,r_1, L^d):=\{f\in H^0(\Delta_r, L^d)\;/\; \Vert f \Vert_{r_1}\leq C^d\Vert f\Vert_{r_0}\}.
\end{equation}

Observe that $B_C(r_0,r_1, L^d)$ is different from zero only if $d\geq 0$ and that, a priori, it is just a set (it may be not closed for the sum).

Since every line bundle on $\Delta_1$ is trivial, if $M$ is another hermitian line bundle on $\Delta$ and $\varphi:L\to M$ is an isomorphism (not necessarily hermitian), then there is a constant $A$ such that, for every $d$,
\begin{equation}\label{differentB}
\varphi^d(B_C(r_o,r_1,L^d))=B_A(r_o,r_1,M^d).
\end{equation}
Consequently, the set $B_C(r_o,r_1,L^d)$ is essentially independent on the chosen metric on $L$. A classical Theorem by Bernstein tells us that $B_C(r_o,r_1,L^d)$ contains the sections of the form $P(z)\cdot e^d$ where $P(z)\in\C[z]$. The Theorem by Sadullaev \cite{SADULLAEV} which will be used in section 8, essentially characterize it.

Let $\mu(\cdot)$ be the Lebesgue measure on $\Delta_r$. For every positive integer $\epsilon$ and global section$f\in H^0(\Delta_1, L^d)$ denote 
$$V(f, r_0,\epsilon):=\mu\left(\{ z\in \Delta_{r_0}\;/\; \Vert f\Vert (z)<\epsilon\Vert f\Vert_{r_0}\}\right).$$

The first lemma is a sort of reduction to the case of polynomials.

\begin{lemma}\label{interpolation} Fix two positive numbers $r_0<r_1<1$. Let $L$ be an hermitian line bundle on $\Delta_1$. Fix a positive constant $C>1$. We can find a positive  constant $C_1=C_1(L,r_0,r_1, C)$ depending only on $L$, $r_0$, $r_1$ and $C$ with the following property: for every $f\in B_C((r_0,r_1, L^d)$, every positive integer $\epsilon$  and every sufficiently big positive integer $\beta$ we have
$$V(f,r_0,\epsilon)\leq C_1\beta d \epsilon^{2/\beta d}.$$

\end{lemma}

\begin{proof} We first recall the following standard fact about interpolation of holomorphic functions on disks, the proof of which can be found in any standard book in complex analysis:
\begin{Fact} \label{interpolationfact} We can find positive constants $C_0$ and $\alpha<1$ depending only on $r_0$ and $r_1$ with the following property: Let $f\in \cO_{\Delta_r}$  and $x_0,\dots, x_N$ be $N+1$ points in $\Delta_{r_0}$. Then there exists a polynomial $P(z)$ for which the following holds:

a) $\deg(P(z))\leq N$;

b) $P(x_i)=f(x_i)$ for every $i=0,\dots N$;

c) $\vert f(z)-P(z)\vert_{r_0}\leq C_0\alpha^{N+1}\vert f\vert_{r_1}$. \end{Fact}  

The line bundle $L$ is of the form $\cO_{\Delta_1}\cdot e$ for a non vanishing global section $e$. We fix a constant $C_1>1$ such that ${{1}\over{C_1}}\leq\Vert e\Vert_{r_1}\leq C_1$. Let $C_0$ and $\alpha$ be the constants defined in the Fact \ref{interpolationfact} above.  We suppose that $\beta$ is such that $C_0\cdot C_1^2\cdot C\cdot \alpha^\beta\leq 1/2$ (such a $\beta$ exists because $\alpha<1$).

Let $f\in B_C((r_0,r_1, L^d)$. We will denote by $V$ the real number $V(f,r_0,\epsilon)$ and by $A$ the set $\{ z\in \Delta_{r_0}\;/\; \Vert f\Vert (z)<\epsilon\Vert f\Vert_{r_0}\}$. We can find points $x_0,\cdots, x_{\beta d}\in \Delta_{r_0}$ such that:

a) $\Vert f\Vert (x_i)\leq \epsilon \Vert f\Vert_{r_0}$;

b) $|x_i-x_j|\geq \sqrt{{V}\over{\pi(\beta d+1)}}$ for $i\neq j$.

For each $j=0;\dots;\beta d$, let $Q_j(z):={{\prod_{i\neq j}(z-x_i)}\over{\prod_{i\neq j}(x_i-x_j)}}.$

Let $P(z)$ be the polynomial for which properties (a), (b) and (c) of Fact \ref{interpolationfact} with respect to the holomorphic function $f$ and points $x_0,\dots, x_{\beta d}$ holds. 

Consider  $P(z)$ as a section of $L^d$ then the following inequality holds:
\begin{equation}\Vert f(z)-P(z)\Vert_{r_0}\leq C_0\cdot C_1^{2d}\cdot \alpha^{\beta d+1}\Vert f\Vert _{r_1}\leq C_0\cdot C_1^{2d}\cdot C^d\cdot\alpha^{\beta d+1}\Vert f\Vert_{r_0}.\end{equation}
The last inequality follows from the fact that $f\in B_C((r_0,r_1, L^d).$

Consequently, from our choice of  $\beta$, we have that:
\begin{equation}\Vert f\Vert_{r_0}\leq 2\Vert P(z)\Vert_{r_0}.\end{equation}

Since the set $\{ Q_j(z)\}$ is a basis of the space of polynomials of degree less or equal then $\beta d$, we can write
\begin{equation}P(z)=\sum_{j=0}^{\beta d} f(x_j)Q_j(z).\end{equation}
Thus we can find a constant $C_2$ such that
\begin{equation}\Vert P(z)\Vert_{r_0}\leq (\beta d+1)\cdot(\pi\cdot (\beta d+1))^{(\beta d)/2}\cdot C_2^d\cdot\epsilon\cdot\Vert f\Vert_{r_0}\cdot\left({{2r_0}\over{\sqrt{V}}}\right)^{\beta d}.\end{equation}
Hence we can find a constant $C_3$ such that
\begin{equation}V\cdot\Vert f\Vert_{r_0}^{2/\beta d}\leq C_3\beta d \epsilon^{2/\beta d}\Vert f\Vert_{r_0}^{2/\beta d}.\end{equation}
The conclusion follows.\end{proof}

A similar Lemma is already proved in \cite{KM}, Proposition 3.2. We reproduce here part of their proof for reader's convenience. 

As a special case of the Lemma we find:

\begin{corollary}\label{areaofpolynomial} We can find a constant $C$, depending on $r$, for which the following holds: For every $\epsilon >0$ and polynomial $P(z)\in\C[z]$ of degree $d$, denoting  by $\mu(\cdot)$ the Lebesgue measure in $\C$, we have 
\begin{equation}\label{eq: measureofpoly} 
\mu\left(\left\{ z\in\Delta_r\;\;/\;\; \Vert P(z)\Vert \leq \epsilon \Vert P(z)\Vert_r\right\}\right)\leq Cd\epsilon ^{2/d}.
\end{equation}
\end{corollary}

\

\section{About a conjecture by Chudnovsky}

In this section we will prove an inequality in the direction  (former) conjecture by Chudnovsky.

If $P(z)\in \C[z_1,\dots , z_N]$ is a polynomial, we will denote by $\Vert P\Vert$ the maximum of the absolute values of its coefficients. 

The function $\Vert P\Vert$ defines a norm on the space of the polynomials. Over the projective space $\P^N$ we may consider the line bundles $\cO_(d)$ equipped with the Fubini-Study metric $\Vert \cdot\Vert_{FS}$. Every polynomial $P$ of degree less or equal then $d$ defines a global section $s^d_P\in H^0(\P,\cO(d))$. One can prove that there exists a constant $C$ such that
$C^d\cdot\Vert P\Vert\leq \Vert s^d_P\Vert_{FS}\leq \Vert P\Vert$. 

In the paper \cite{AMOROSO} the author proves that the set of points $\zeta\in\C^N$ for which we can find a constant $A$, such that, for every non zero polynomial $P(z)\in \Z[z_1,\dots , z_N]$ of degree $d$, we have $\log\vert P(\zeta)\vert\geq -A(\log\Vert P\Vert +d)^{N+1}$, is  of full measure  in $\C^N$. This was conjectured by Chudnovsky in \cite{CHUDNOVSKY}. 

In this section we will prove a similar inequality, in the same contest. The inequality we prove is independent of the one proved by Amoroso. It is weaker in the degree of the polynomial but stronger in the height of it. 

\begin{definition}\label{S points on CN}
Suppose that  $a\geq N+1$.  A point $\zeta\in \C^N$ is said to be of type $S^N_a(\C)$ if we can find  a positive constant $A$ depending on $\zeta$, such that, for every non zero polynomial $P(z)\in \Z[z_1,\dots , z_N]$ of degree $d$ we have
\begin{equation}\label{eq: SNforpoly}
\log\vert P(\zeta)\vert\geq -A d(d^{a-1}\log\Vert P\Vert +d\log(d)+1).
\end{equation}
\end{definition}

Given a number $a\geq N+1$ and  a point $z$ of type $S_a^N(\C)$,  the minimal $A$ for which the definition \ref{S points on CN} holds for $Z$ and $\cL$ can be called the {\it transcendental height} of $z$ with respect to the couple ($\cL,a$). This would give an analogy with the corresponding inequality for algebraic points.

The main theorem of this section is the following:
\begin{theorem}\label{amoroso1} For every $a\geq N+1$, the set of points of type $S^N_a(\C)$ is full in $\C^N$.
\end{theorem}

Before we start the proof we recall the following fact:

-- Denote by $V_d(N)$ the lattice of polynomial of degree at most $d$ in $N$ variables with integer coefficients. Then we can find constants $B_i$ (depending only on $N$, but independent on $d$) such that, for every positive constant $H$ we have:

\begin{equation}
\Card\left(\left\{ P(z)\in V_d(N) \;\;/\;\; \Vert P(z)\Vert\leq H\right\}\right)\leq B_1H^{B_2d^N}.
\end{equation}

\begin{proof} It suffices to prove the theorem when $a=N+1$.

The proof of the Theorem is by induction on $N$. The case of $N=1$ is slightly different then the general case. We will explain the differences during the proof. 

We suppose that the Theorem is true for $\C^{N-1}$: the  the set of points of type $S_N^{N-1}(\C)$ on $\C^{N-1}$ is full on $\C^{N-1}$.  

Fix $\zeta=(\zeta_2,\dots, \zeta_N)\in\C^{N-1}$.  If $P(z_1,\dots, z_N)\in\Z[z_1,\dots, z_N]$, we denote by $P_\zeta(z)\in\Z[\zeta_2, \dots, \zeta_N][z]$ the polynomial $P(z,\zeta_2,\dots,\zeta_N)$. If $\deg(P(z))=d$ then $\deg(P_\zeta(z))\leq d$ (remark that $P_\zeta[z]$ is a polynomial in $\C[z]$). When $N=1$ then put $P_\zeta(z)=P(z)$.

Suppose that $\zeta\in\C^{N-1}$ is an arithmetically generic point (no polynomial with coefficient in $\Z$ vanish on it). Fix an $r>0$. A complex number $z\in\Delta_r$ is said to be of type $S_N(\zeta)$ if there exists a constant $A=A(z,\zeta)$ such that, for every $P\in\Z[z_1,\dots,z_N]$ of degree at most $d$ we have
\begin{equation}\label{inductiveinequality}
\log|P_\zeta(z)|\geq -Ad^{N+1}(\log\Vert P\Vert+d^2\log(d))+\log\Vert P_\zeta\Vert_r
\end{equation}
(We recall that $\Vert P\Vert_r=\sup_{z\in\Delta_r}\{ |P(z)|\}$). 
We claim that it suffices to prove the following:

{\bf Claim}: For every arithmetically generic $\zeta\in \C^{N-1}$ and every positive $r>0$, the set of $z\in \Delta_r$ of type $S_N(\zeta)$ is full in $\Delta_r$.

Indeed,  by Fubini Theorem, the claim implies that the subset  of points  $(z,\zeta)\in\Delta_r\times \C^{N-1}$ for which there exists a constant $A$ (depending on $(z,\zeta)$) such that, for every $P\in \Z[z_1\dots,z_N]$ inequality \ref{inductiveinequality} holds, is full. 

Suppose now that $\zeta\in\C^{N-1}$ is of type $S_{N}^{N-1}(\C)$ (or, if $N=1$, we suppose that $\zeta=0$).  Let $P\in \Z[z_1\dots,z_N]$ of degree $d$. Suppose, for the moment that $z_1\nmid P$. Inductive hypothesis implies that 
\begin{equation}\label{maxofpzeta}
\log\Vert P_\zeta\Vert_r\geq \log\Vert P(0,\zeta)\Vert_\geq -B_1d^{N}\log\Vert P\Vert -B_2d^2\log(d)-B_3d
\end{equation}
for suitable constants $B_i$ depending only on $\zeta$. When $N=1$ then observe that $\log\Vert P_\zeta\Vert_r\geq 0$.  If $z_1\mid P$, consider $P_1:={{P}\over{z_1^a}}$ for a suitable $a$. An inequality similar to \ref{maxofpzeta} holds for $P_1$ and the maximum modulus principle implies that \ref{maxofpzeta} holds (in this case the constant $B_3$ depends on $r$).

Consequently, if $\zeta\in\C^{N-1}$ is of type $S^{N-1}_{N}(\C)$ and $z$ is of type $S_N(\zeta)$, then $(z,\zeta)$ is of type $S_{N+1}^N(\C)$. Thus the claim implies the theorem.

We now prove the claim. Let $P(z_1,\dots, z_N)$ a non zero polynomial of degree $d$. We apply now formula \ref{eq: measureofpoly} to $P_\zeta(z)$ and $\epsilon={{1}\over{\Vert P\Vert_r^{B_1d^{N+1}}}d^{B_2d^2}}$, where the $B_i$'s will be chosen after. By Corollary \ref{areaofpolynomial} we obtain then that we can find a constant $C$ depending only on $r$ such that, for every polynomial $P\in\Z[z_1,\dots, z_N]$ of degree $d$, we have
\begin{equation}\label{areaestimate2}
\mu\left(\left\{ z\in \Delta_r\;/\; \vert P_\zeta(z)\vert\leq {{\Vert P_\zeta\Vert_r}\over{\Vert P\Vert^{B_1d^{N+1}}d^{B_2d^2}}}\right\}\right)\leq{ {C}\over{\Vert P\Vert^{B_1d^N}d^{B_2d}}}.
\end{equation}

Denote by $A(P)$ the area of the set $\{z\in \Delta_r\;/\; \vert P_\zeta(z)\vert\leq {{\Vert P_\zeta\Vert_r}\over{\Vert P\Vert^{B_1d^{N+1}}d^{B_2d^2}}}\}$.

In order to conclude,  we compute
\begin{eqnarray*}
&\sum_{d=1}^\infty\sum_{P\in V_d(N)}A(P)&\leq\\
&&\leq \sum_{d=0}^\infty\sum_{P\in V_d(N)} {{B_0}\over
{\Vert P\Vert^{B_1d^{N}}d^{B_3d}}} \\
&& =\sum_{d=1}^\infty\sum_{H=1}^\infty\sum_{  {P\in V_d(N)}\atop
{H\leq \Vert P\Vert_<H+1}} {{B_0}\over
{\Vert P\Vert^{B_1d^{N}}d^{B_3d}}}\\
&&\leq \sum_{d=1}^\infty\sum_{H=1}^\infty{{B_0H^{B_4d^N}}\over{H^{B_1d^N}d^{B_3d}}}
\end{eqnarray*}
where the $B_i$'s are sufficiently big constants. We are free to choose the constants $B_1$ and $B_3$. It is easy to see that the last sum converges as soon as $B_1> B_4$ and $B_3>1$. By Borel--Cantelli Lemma we conclude.\end{proof}

\subsection{Theorem \ref{amoroso1} over a $p$--adic field}\label{padicfields}
The proof of Theorem \ref{amoroso1} do not use any specific properties of the complex numbers, thus it can be generalized, for instance, to any $p$--adic field. We give here some details on the (few) changes one can do.

Let $K$ be a finite extension of  $\Q_p$.  It is a locally compact fields. We equip it with the unique norm such that $\Vert p\Vert_p={{1}\over{p}}$. We fix a positive integer $N\geq 1$. There is a unique  translation invariant measure $\mu_p(\cdot)$ on $K^N$ such that $\mu_p(\{z\in K^n\;/\;\Vert z\Vert_p\leq 1\}=1$.  For this measure we have that
$\mu_p(\{ z\in K^N\;/\;\Vert z\Vert_p \leq p^\alpha\})=p^{N\alpha}$. As before, if $P(z)\in K[z]$ is a polynomial and $r>0$ is a real number contained in the set $\Vert K\Vert_p:=\{ \Vert a\Vert_p\;/\; a\in K\}$, we will denote by $\Vert P\Vert_r$ the number $\sup\{ \Vert P(z)\Vert_p\;/\; \Vert z\Vert_p \leq r\}$. 

A definition similar to definition \ref{S points on CN} can be given:

\begin{definition}\label{S points on Qp}
Suppose that  $a\geq N+1$.  A point $\zeta\in K^N$ is said to be of type $S_a^N(K)$ if we can find  a positive constant $A$ depending on $\zeta$, such that, for every non zero polynomial $P(z)\in \Z[z_1,\dots , z_N]$ of degree $d$ we have
\begin{equation}\label{eq: SNforpolyQ}
\log\Vert P(\zeta)\Vert_p\geq -A d(d^{a-1}\log\Vert P\Vert +d\log(d)+1).
\end{equation}
\end{definition}

Observe that on the left of the inequality \ref{eq: SNforpolyQ} the norm is the the $p$--adic norm, and on the right we have the norm of the polynomial with integer coefficients given at the beginning of this section.

\begin{theorem}\label{padicamoroso} The set of points of type $S_{N+1}^N(K)$ is full in $K^N$ (the complementary has zero measure in $K^n$). 
\end{theorem}

If $r\in\Vert K\Vert_p$, we will denote $\Delta_r(K)$ the "disk" $\{ z\in K\;/\; \Vert z|\Vert_p\leq r\}$. An analogous of Corollary \ref{areaofpolynomial} holds in this contest:

\begin{proposition}\label{padicareaofpolynomial} For every $\epsilon >0$ and polynomial $P(z)\in K[z]$ of degree $d$, we have 
\begin{equation}\label{eq: padicmeasureofpoly} 
\mu\left(\left\{ z\in\Delta_r(K)\;\;/\;\; \Vert P(z)\Vert_p \leq \epsilon \Vert P(z)\Vert_r\right\}\right)\leq r\epsilon ^{1/d}.
\end{equation}
\end{proposition}
\begin{proof} The proof is identical to the proof of Lemma \ref{interpolation}. We give here just a sketch of the argument: Let $P(z)\in K[z]$ be a polynomial of degree $d$. Denote  $B$ the set defined in the proposition and by $V$ its measure. We can find $x_0,\dots x_d$ in  $B$ such that $\Vert x_i-x_j\Vert_p\geq {{V}\over{d+1}}$. Denote by  $Q_j(z)$ the polynomials ${{\prod_{i\neq j}(z-x_i)}\over{\prod_{i\neq j}(x_i-x_j)}}.$. Then $P(z)=\sum P(x_i)Q_j(z)$ and using the strong triangular inequality, we obtain $\Vert P(z)\Vert_r\leq \epsilon \Vert P(z)\Vert _r\cdot{{r^d}\over{V^d}}$. The conclusion follows. \end{proof}

The proof of Theorem \ref{padicamoroso} is now identical to the proof of Theorem \ref{amoroso1}, replacing Corollary \ref{areaofpolynomial} by Proposition \ref{padicareaofpolynomial}.

A similar argument will hold over $\R^N$.

\section{Arithmetically generic sub--varieties}\label{arithmetically generic sub varieties}

In the sequel we will study possible generalizations of the Liouville inequalities on transcendental points on varieties. At the moment the inequalities we find are not good enough to be applied to obtain theorems as in the previous section, but we hope that a further analysis will give possible applications.

We fix a smooth projective variety $X_K$ defined over $K$ and an arithmetic polarization $(\cX, \overline\cL)$ over it. In this section we will focus our attention on closed sub varieties of $X_{\sigma_0}$ which are far from being defined over a number field.

\begin{definition} \label{genericsubvariety} Let $Z\subset X_{\sigma_0}(\C)$ be a Zariski closed sub variety. We will say that $Z$ is {\rm arithmetically generic}  (or that $Z$ is not defined over the algebraic closure of $K$) if, for every positive integer $d$ the natural map
$$H^0(\cX,\cL^d)\longrightarrow H^0(Z, (\cL^d)_{\sigma_0}|_Z)$$
is injective. 
\end{definition}

It is easy to verify that this definition is independent on the choice of the arithmetic polarization. 

In general a complex sub variety $Z\subset X_{\sigma_0}$ is said to be defined over $\overline{K}$ if it is obtained by base change to $\C$ of a variety $Z_{\overline K}\subset X_{\overline K}:=X_K\otimes\overline{K}$ (A variety over an algebraically closed field is an irreducible scheme of finite type over the field). In general there are many sub varieties of $X_{\sigma}$ which are not defined over $\overline K$ which and are {\it not} arithmetically generic: it suffices to consider for instance a transcendental point on a curve defined over $K$. 

The definition means that $Z$ cannot be contained in a proper sub variety defined over the algebraic closure of $K$. Thus, being arithmetically generic is a property which is stronger then "not being defined over the algebraic closure of $K$".  For instance, if $Z$ is a point in a projective space, $Z$ would be arithmetically generic if the coordinates of it are algebraically independent. Observe that, with this definition, $X_{\sigma_0}$ as a sub variety of itself, {\it is not defined over the algebraic closure of $K$} (or it is arithmetically generic). This is a bit unnatural (and not correct) but we prefer to keep the definition like that because this will only affect the following fact:  if $Z$ is a sub variety of $X_{\sigma_0}$not defined over the algebraic closure of $K$, we may suppose that $Z$ is $X_{\sigma_0}$ itself. 

\smallskip

We now list some general properties of arithmetic generically sub varieties.

The set of arithmetically generic sub varieties of $X_\sigma$ is independent of the chosen polarization. In order to prove this, it suffices to remark that $Z\subset X_\sigma$ is arithmetically generically if and only if {\it for every} effective divisior $D_K\subset X_K$, the restriction of $D_K$ to $Z$ is not zero. More details on the independence of the polarization will be given in Section \ref{Sapoints} for a special class of arithmetically  generic sub varieties. 

\begin{lemma}\label{countable1} Let $Z$ be a Zariski closed sub variety of $X_{\sigma_0}(\C)$ which is arithmetically generic. Then the set of irreducible effective Cartier divisors of $Z$ which are {\rm not} arithmetically generic, is countable. \end{lemma}

\begin{proof} Let $D$ be an irreducible effective Cartier divisor of $Z$ which is defined over the algebraic closure of $K$. This means that we can find a positive integer $d$ and a section $s\in H^0(\cX, \cL^d)$ which vanishes along $D$.  Thus $\div(s|_Z)=\sum n_iD_i + n_0D$. By hypothesis, $s$ do not vanish identically on $Z$.
For every integer $d$ and for every section $s\in H^0(\cX,\cL^d)$ denote by 
$DIV(Z,s)$ the set $\{ D_i\in Div(Z) \;/\;\div(s|_Z)=\sum n_iD_i\}$. The set of if irreducible Cartier divisors of $Z$ which are defined over the algebraic closure of $K$ is the union of the $DIV(Z,s)$'s where $s$ runs over $s\in H^0(\cX, \cL^d)$ and $d$ runs over the positive integers. Since each of the $H^0(\cX, \cL^d)$ is countable and each of the $DIV(Z,s)$ is finite, the conclusion follows. \end{proof}

As a consequence of this, we find:

\begin{proposition}
 There exist sub varieties of any dimension which are arithmetically generic. 
 \end{proposition}
 \begin{proof} Let $D$ be an effective divisor of $X_{\sigma_0}$. The complementary of it is a open set, which we denote by $U_D$, dense in $X_{\sigma_0}(\C)$ for the euclidean topology.  A point  $z\in X_{\sigma_0}(\C)$  is arithmetically generic if it do belong  to $U_{\div(s)}$ for every possible $s\in H^0(\cX,\cL^d)$ and integer $d$.  Thus the set of points $z\in X_{\sigma_0}(\C)$ which are arithmetically generic  is the intersection of countably many $U_D$'s. From Baire Lemma we deduce that these points are dense in $X_{\sigma_0}$. Every closed variety passing through at least one of these points is not defined over the algebraic closure of $K$. \end{proof}
 
 We keep in mind the following corollary:
 \begin{corollary}\label{genericbyspecialization} Suppose that $Z_1\subset Z_2\subset X_{\sigma_0}(\C)$ are Zariski closed sub varieties. If $Z_1$ is arithmetically generic then $Z_2$ is arithmetically generic.\end{corollary}
 
 In particular, if $Z$ contains an arithmetically generic point, then itself is arithmetically generic.

In this paper we would like to deal with properties which are shared by "almost" all the varieties. We give a definition which clarify the notion of "almost all the varieties": 

\begin{definition}\label{almost all} Let $R$ be a set of sub varieties of $X_{\sigma_0}$. We will say that $R$ is full if, for every fibration $f:X_K\to B_K$, the set of $b\in B_{\sigma_0}(\C)$ such that $X_b:=f^{-1}(b)$ is in $R$ is either empty or full in $B_{\sigma_0}$ (for the Lebesgue measure).
\end{definition}

With this definition, as a consequence of the proposition below,  the set of arithmetically generic varieties is full:

\begin{proposition}\label{fibration and generic} Let $f:X_K\to B_K$ a fibration. Then a point $b\in B_{\sigma_0}(\C)$ is arithmetically generic in $B_K$ if and only if the fibre $X_b:=f^{-1}(b)$ is an arithmetically generic  variety in $X_K$. Moreover if $X_K$ and $B_K$ are smooth, if $b\in B_{\sigma_0}(\C)$ is arithmetically generic then $X_b$ is smooth.\end{proposition}
\begin{proof} We first make the following observation: suppose that $Z\subset X_{\sigma_0}$ is a Zariski closed sub variety and $M_K$ is {\it any} line bundle over $X_K$ (thus defined over $K$). Suppose that $s\in H^0(X_K,M)$ is a global section such that $s|_Z=0$, then Z is not arithmetically generic (proof: $\div(s)$ is a component of an ample divisor of $X_K$).

As a consequence of the observation we have: Suppose that $X_b$ is arithmetically generic, then $b$ must be arithmetically generic. Indeed, if not, then there is a divisor $D$ in $B$ such that $D|_b=0$. Consequently, $f^\ast(D)|_{X_b}=0$ and thus $X_b$ is not arithmetically generic. 

Another consequence of the observation is  the following: suppose that $f$ is generically finite and that the fibre $X_b$ is finite. Suppose that $b\in B_{\sigma_0}(\C)$ is arithmetically generic, then every element $x\in X_b(\C)$ is arithmetically generic. Indeed, if $D$ is an effective divisor of $X_K$ passing through $x$, then $f_\ast(D)$ is an effective divisor of $B_K$ passing through $f(x)$. 

Suppose that $b\in B_{\sigma_0}(\C)$ is arithmetically generic. Let $Y\subset X_K$ be a closed sub variety such that $Y\cap X_b$ is finite. Then every point of $Y\cap X_b$ is arithmetically generic, and consequently $X_b$ is arithmetically generic.

For the last part of the Proposition: by Bertini Theorem, there is a Zariski open set $U_K\subset B_K$ such that $f|_{U_K}:X_K|_{U_K}\to U_K$ is smooth. If $b$ is an arithmetically generic point, then $b\in U_K(\C)$ because the complementary of $U_K$ is a closed sub variety defined over $K$. 
\end{proof}

We will now show that, for arithmetically generic varieties inequalities of Liouville type as in Theorem \ref{liouville2} cannot hold. 

\begin{theorem}\label{liouville3}We can find a constant $A$ depending only on the chosen arithmetic polarization for which the following holds:  let $z\in X_{\sigma_0}(\C)$ be an arithmetically generic point, then, for every $d\in \bN$ there exists an infinite sequence of sections $s_n\in H^0(\cX,\cL^d)$ for which
\begin{equation} \log\Vert s_n\Vert_{\sigma_0}(z)\leq -Ad^{\dim(X_K)}(\log^+\Vert s_n\Vert +d).
\end{equation}
\end{theorem}

\begin{proof} We denote by $L$ the line bundle $\cL_{\sigma_0}$. For every positive integer $d$, let $L_z^d$ be the fibre of $L^d$ over $z$. We put on $L_z^d$ the unique Haar measure $\mu$ for which $\mu(\{ v\in L_x^d \; / \Vert v\Vert_{\sigma_0}\leq 1\})=1$.   For every positive real number $T$, we have that $\mu(\{ v\in L_x^d \; / \Vert v\Vert_{\sigma_0}\leq T\})=T^2$. 

Since $z$ is an arithmetically generic point, we have an inclusion $\iota_d :H^0(\cX, \cL^d)\hookrightarrow L_z^d$. By Dirichlet box principle and inequality \ref{eq:counting1} we can find positive constants $C_i$ for which the following holds: as soon as $T$ is sufficiently big, there are two distinct sections $s_1$ and $s_2$ in $H^0(\cX, \cL^d)$ such that
$\sup_{\tau\in M_K^\infty}\{\Vert s_i\Vert_\tau\}\leq T$ and $\Vert s_1-s_2\Vert_{\sigma_0}(z)\leq {{T^2}\over{C_1^{d^{\dim(X_K)+1}}T^{C_2d^{\dim(X_K)}}}}$. The conclusion follows from the fact that, if $T>1$, then 
$\sup_{\tau\in M_K^\infty}\{1, \Vert s_i\Vert_\tau\}\leq T$.\end{proof}

Following the path of the proof of the previous theorem we can prove the following
\begin{theorem}\label{liouville4} Let $Y\subset X_{\sigma_0}$ be an arithmetically generic sub variety. Then we can find a constant $A$ depending on $Y$ and on the arithmetic polarization  for which the following holds:  for every $d\in \bN$ there exists an infinite sequence of sections $s_n\in H^0(\cX,\cL^d)$ for which
\begin{equation} \log\Vert s_n\Vert_{Y;\sigma_0}(z)\leq -Ad^{\dim(X_K)}(\log^+\Vert s_n\Vert +d).
\end{equation}
\end{theorem}

\section{$S_a$ points on projective varieties}\label{Sapoints}

Let $X_K$ be a smooth projective variety of dimension $N$ defined over $K$.  Fix an arithmetic polarization $(\cX,\cL)$ of $X_K$. In this section we will define a class of arithmetically generic points on $X_{\sigma_0}(\C)$ which verify a kind of weak Liouville inequality. This class is quite natural and  in the next sections we will show that it is a set of full Lebesgue measure.

Suppose that $z\in X_{\sigma_0}(\C)$ is a complex point.

\begin{definition} \label{SNpoints}  Let $a\geq N$ be a real number. We will say that $z$ is of type $S_a(\cX)$ if we can find a positive constant $A=A(z,\overline\cL, a)$, depending on $z$, the polarization  and $a$, such that, for every positive integer $d$ and every non zero global section $s\in H^0(\cX, \cL^d)$ we have that
\begin{equation}\label{liouville Sa points}
\log\Vert s_{\sigma_0}\Vert_{\sigma_0}(z)\geq -Ad^{a}(\log^+\Vert s\Vert +d).
\end{equation}

We will denote by $S_a(X_K)$ the subset of $X_{\sigma_0}(\C)$ of points of type $S_a(\cX)$. 
\end{definition}

In analogy with Remark \ref{Liouville and heights}, when $z\in S_a(X_K)$, there is a minimal value $A_{\cL,a}(z)$ within the $A$'s for which  the inequality \ref{liouville Sa points} holds. In this case we can say that  the transcendental height of $z$  with respect to $\cL$ and $a$, is  $A_{\cL,a}(z)$.

Observe that a point of type $S_a(\cX)$ is necessarily arithmetically generic.  

Remark also that the condition $a\geq N$ is necessary because of Theorem \ref{liouville3}.

We fix $a\geq N$.

The set $S_a(X_K)$ do not depend of the arithmetic polarization. This is proved in the following proposition. 

\begin{proposition}\label{independence1} The set $S_a(X_K)$ is independent on the choice of the arithmetic polarization.
\end{proposition}
\begin{proof} We begin by fixing an arithmetic polarization on $X_K$. Let $z\in S_a(X_K)$. For the time being we will say that $z$ is of type $S_a(X_K)$ with respect to the fixed arithmetic polarization. 

First remark that, if $\cL$ is the involved ample line bundle on the involved model $\cX$ and $d_0$ is a positive integer, then, for every positive integer $d$ and every global section $s\in H^0(\cX, \cL^{d_0d})$ we have that
\begin{equation}
\log\Vert s_{\sigma_0}\Vert_{\sigma_0}(z)\geq -(Ad_0)^{a}d^{a}(\log^+\Vert s\Vert +d_0\cdot d).
\end{equation}
Conversely, suppose that we can find a constant $A_1$  such that, for every positive integer $d$ and every  $s\in H^0(\cX, \cL^{d_0d})$ we have that $\log\Vert s_{\sigma_0}\Vert_{\sigma_0}(z)\geq -A_1d^a(\log^+\Vert s\Vert +d)$. Consequently, if $s\in H^0(\cX, \cL^d)$, then, $s^{d_0}\in H^0(\cX, \cL^{d_0d})$. Since $\log\Vert s^{d_0}\Vert_{\sigma_0}(z) = d_0\log\Vert s\Vert_{\sigma_0}(z)$ and $\log^+\Vert s^{d_0}\Vert\leq d_0\log^+\Vert s\Vert$, we have that
\begin{equation}
\log\Vert s_{\sigma_0}\Vert_{\sigma_0}(z)\geq --A_1d^{a}(\log^+\Vert s\Vert +d).
\end{equation}
Hence, changing $\cL$ with $\cL^{d_0}$ do not affect the set $S_a(X_K)$. 

{\it Independence on the the change of metrics}: For every $\sigma\in M_K^\infty$, denote by $\Vert\cdot\Vert_\sigma$ the metric on the involved very ample line bundle  $L_\sigma$. Suppose that we change the metric on $L_\tau$ with $\tau\neq \sigma_0$. Denote by $\Vert\cdot\Vert^1_\tau$ the new metric. Since $X_\tau$ is compact, we can find a constant $C$ such that, for every positive integer $d$ and every global section $s\in H^0(\cX, \cL^d)$, we have $\log^+\Vert s\Vert^1\leq \log^+\Vert s\Vert+Cd$, Hence changing the metric at a place different from $\sigma_0$ do not affect $S_a(X_K)$. Suppose now that we change the metric on $\cL_{\sigma_0}$. Denote again $\Vert\cdot\Vert^1_{\sigma_0}$ the new metric. We can find constants $C_1$ and $C_2$ such that, for every positive integer $d$ and every global section $s\in H^0(\cX, \cL^d)$, we have $\log^+\Vert s\Vert^1\leq \log^+\Vert s\Vert+C_2d$, and $\Vert s_{\sigma_0}\Vert_{\sigma_0}(z)\geq \Vert s_{\sigma_0}\Vert_{\sigma_0}^1(z)+C_2d$. Thus the change of the metric at the place $\sigma_0$ do not affect $S_a(X_K)$.

{\it Independence on the involved ample line bundle $\cL$}. Suppose that $\cM$ is a relatively ample line bundle on the model $\cX$. Replacing $\cL$ and $\cM$ by a positive power if necessary, we may suppose that there exists an effective divisor $D$ on $\cX$ such that $\cL=\cM(D)$. We may choose metrics $\Vert\cdot\Vert^\cL$, $\Vert\cdot\Vert^\cM$ and $\Vert\cdot\Vert^D$ on $\cL$,  $\cM$ and $\cO_\cX(D)$ respectively in such a way that this is an isometry (these changes will not affect the points $S_a(X_K)$ with respect to both arithmetic polarizations). Denote by $s_D\in H^0(\cX,\cO_\cX(D)$ a section such that $div(s_D)=D$. Observe that $s_D(z)\neq 0$. 

Suppose that $s\in H^0(\cX, \cM^d)$ is a non zero section. Then $\tilde s:=s\otimes s_D^d\in H^0(\cX; \cL^d)$.
If we denote by $B_1$ the real number $\log\Vert s_{D.\sigma_0}\Vert_{\sigma_0}(z)$, we have that $\log\Vert \tilde s_{\sigma_0}\Vert^\cL_{\sigma_0}(z)=\log\Vert s_{\sigma_0}\Vert^\cM_{\sigma_0}(z)+dB_1$. Moreover we can find a constant $B_2$ such that $\log^+\Vert\tilde s\Vert \leq \log^+\Vert s\Vert +dB_2$. Hence if $z$ is of type $S_a(\cX)$ with respect to $\cL$ then $z$ is of type $S_a(\cX)$ with respect to $\cM$. Exchanging the roles of $\cL$ and $\cM$ the independence on $\cL$ follows. 

{\it Independence on the involved model $\cX$}. Suppose that we have two models $\cX_1$ and $\cX_2$ of $X_K$. We may suppose that $\cX_1$ is  a blow up of $\cX_2$ with exceptional divisor $E$. Thus we have a birational morphism $f:\cX_1\to \cX_2$ which may supposed to be the identity on the generic fiber. We may suppose that $\cL_2$ is an ample line bundle on $\cX_2$ such that $\cL_1:=f^\ast(\cL_2)(-E)$ is ample on $\cX_1$. Let $z\in X_{\sigma_0}(\C)$.  The fact that if $z$ is of type $S_a(X_K)$ with respect to $\cL_1$ then it is of type $S_a(X_K)$ with respect to $\cL_2$ follows the same path of the proof above (independence on the line bundle) thus it is left to the reader. 

Observe that $E$ is a vertical divisor, thus there is an integer $N$ with the following property: Let $dE$ be the $d$-th infinitesimal neighborhood of $E$; suppose that $\cM$ is a line bundle on $\cX_1$ and $g\in H^0(dE, \cM|_{dE})$ then $N^d\cdot g=0$. This is proved by induction on $d$ by using the exact sequence
$$0\longrightarrow\cM(-dE)\vert_{E}\longrightarrow \cM |_{dE}\longrightarrow  \cM |_{(d-1)E}\longrightarrow 0.$$

For every positive integer $d$, we  have an exact sequence on $\cX_1$
$$0\longrightarrow f^\ast( \cL_2^d)\longrightarrow \cL^d_1\longrightarrow \cL_1^d|_{dE}\longrightarrow 0.$$
Hence, if $s\in H^0(\cX_1, \cL_1^d)$ then $N^d\cdot s\in H^0(\cX_2, \cL_2^d)$. Consequently, if z is of type $S_a(X_K)$ with respect to $\cL_2$ then it is of type $S_a(X_K)$ with respect to $\cL_1$. 
\end{proof}

As a consequence of proposition above, it is correct to do not mention the arithmetic polarization when we consider the set $S_a(X_K)$.

\begin{remark}It is interesting to remark that the proof of the independence of $S_a(X_K)$ on the polarization is very similar to the proof, in Arakelov geometry, of the independence of the height definition on the choice of models and metrics. 
\end{remark}

A partial analogue of Proposition \ref{fibration and generic} holds:

\begin{proposition}\label{fibrations and Sa} Let $f: X_K\to B_K$ be a  morphism between projective varieties defined over $K$. Let $z_0\in S_a(X_K)$. Then $f(z_0)\in S_a(B_K)$. 
\end{proposition}
\begin{proof} We can fix arithmetic polarizations $(\cX, \cL)$ and $(\cB, \cH)$ of $X_K$ and $B_K$ respectively. We may also suppose that  the fibration extends  to a projective morphism $f:\cX\to \cB$ and there exists an effective  Cartier divisor $\cV$ on $\cV$ such that $\cL=f^\ast(\cH)+\cV$. Let $s_V\in H^0(\cX, \cO(\cV))$ such that $\div(s_V)=\cV$. Remark that, since $z_0\in S_a(X_K)$,  it is arithmetically generic and consequently $z_0\not\in \cV_\sigma(\C)$. 

Let $d$ be a positive integer and  $s\in H^0(\cB,\cH^d)$. Then,  $f^\ast(s)\cdot s_V^d\in H^0(\cX,\cL^d)$. Since $z_0\in S_a(X_K)$, we have that, for a suitable constant $A$,
\begin{equation}
\log\Vert s\cdot s^d_V\Vert(z_0)\geq -Ad^n(\log^+\Vert f^\ast(s)\cdot s^d_V\Vert +d).
\end{equation}
But since $\log\Vert f^\ast(s)\cdot s^d_V\Vert (z_0)=\log\Vert s\Vert (f(z_0))+d\log\Vert s_V\Vert(z_0)$ and $\log^+\Vert f^\ast(s)\cdot s^d_V\Vert\leq \log^+\Vert s\Vert +d\log^+\Vert s_V\Vert$. The conclusion follows. \end{proof}

It would be interesting to prove a converse of Proposition \ref{fibrations and Sa}. If such a converse is true, a big part of the theory developed in this paper would be reduced to the case $X_K=\P^N$.

\section{Arithmetically generic curves}

In this section we would like to study properties analogues to the property $S_a(\cX)$  on points which belong to the same arithmetically generic curve (Zariski sub variety of dimension one). 

At the moment we are not able to attack the case when $N\leq a <N+1$. Thus, from now on we are going to suppose that 
\begin{equation}
a\geq N+1.
\end{equation}

Again  we fix a smooth projective variety $X_K$ defined over $K$ and an arithmetic polarization $(\cX, \overline\cL)$ over it. We also fix an arithmetically generic closed curve $Y\subset X_{\sigma_0}(\C)$.

Again, if  $s\in H^0(\cX,\cL^d)$ is a non zero section, we will denote $$\log\Vert s\Vert_Y:=\sup\{\log\Vert s\Vert_{\sigma_0}(z) \;/ \; z\in Y\}.$$

Remark that $\log\Vert s\Vert_Y$ is a real number, because, since $Y$ is arithmetically generic,  none non zero section $s\in H^0(\cX,\cL^d)$ will vanish identically on $Y$. 

In analogy with the definition \ref{SNpoints} we give:

\begin{definition} \label{SNpoints2} We will say that $z\in Y$ is of type $S_a^Y$ (or that $z\in S_a(Y)$) if we can find positive a constant $A=A(z,\overline\cL,Y, a)$ depending on $z$, the polarization, $Y$  and $a$ such that, for every positive integer $d$ and every non zero global section $s\in H^0(\cX, \cL^d)$ we have that
$$
\log\Vert s_{\sigma_0}\Vert_{\sigma_0}(z)\geq -Ad^{a}(\log^+\Vert s\Vert +d)+\log\Vert s\Vert_Y.
$$
Moreover we will denote by $S_a(Y)$ the subset of $Y$ of points of type $S^Y_a$. 
\end{definition}

Remark that the main difference between points of $S_N(X_K)$ and points of $S_a(Y)$ is in the last term on the right of the inequality. A priori $\log\Vert s\Vert_Y$ may be much smaller then $\log^+\Vert s\Vert$ (besides the fact that we are not considering the $\log^+(\cdot)$).

The same proof (mutatis mutandis) of Proposition \ref{independence1} gives:

\begin{proposition} The set $S_a(Y)$ is do not depend on chosen the arithmetic polarization.
\end{proposition}

The main theorem of this section is the following:

\begin{theorem}\label{SNongenericcurves} With the notations as above, then the set $S_a(Y)$ is full in $Y$.
\end{theorem}

An important corollary of this theorem is 
\begin{corollary}\label{fullSNoncurves} If $Y$ is an arithmetically generic curve then $Y\cap S_a(X_K)\neq\emptyset$ if and only if  $S_a(X_K)\cap Y$ is full in $Y$.\end{corollary}

Thus, we find the following important principle:

\
 
{\it Either an arithmetic generic curve is contained in the complementary of $S_a(X_K)$, or its intersection of it with $S_a(X_K)$ is almost all the curve itself. }

\

In order to prove Theorem \ref{SNongenericcurves}, we may suppose that $a=N+1$.

The proof of Theorem \ref{SNongenericcurves} requires three lemmas of different nature. The first needed Lemma is Lemma \ref{interpolation}.

The second Lemma is a consequence of the following Theorem due to Sadullaev (we present it here in the form we need):

\begin{theorem}\label{sadullaev1} Let $Z\subset \C^n$ be an analytic sub variety of dimension one. Let $U_1\subset U_2$ be two relatively compact open sets of $Z$. Then the followings are equivalent:

a) $Z$ is an open set of an affine curve (that means that $Z$ is algebraic);

b) There exists a constant $C$ (depending only on the $U_i$'s and $n$) with the following property: For ever polynomial $P(z)=P(z_1,\dots , z_n)\in \C[z_1;\dots, z_n]$ we have that
$$\sup_{z\in U_2}\{ |P(z)|\}\leq C^{\deg(P)}\sup_{z\in U_1}\{ |P(z)|\}.$$
\end{theorem}

For a proof, cf. \cite{SADULLAEV} Theorem 2.2.

From theorem \ref{sadullaev1} we deduce the following:

\begin{lemma}\label{sadullaev11} Let $Y$ be a compact Riemann surface equipped with an hermitian line bundle $L$. Let $U\subset Y$ be a (non compact) simply connected open set whose complementary has non empty interior. Fix a biholomorphic isomorphism $\varphi_U:\Delta_1\to U$. Fix two positive integers $0<r_0<r_1<1$. Then we can find a constant $C_1$ such that, for every positive integer $d$ and every $s\in H^(Y, L^d)$ we have that
\begin{equation}\varphi_U^\ast(s)\in B_{C_1}(r_0,r_1,L^d|_U).\end{equation}
\end{lemma}

The definition of  $B_{C_1}(r_0,r_1,L^d_U)$ is given in \ref{definition of B}.

\begin{proof} We first claim that, if there exists  an hermitian line bundle $M$ and an effective divisor $D$ on $Y$ such that $M=L(D)$ (we fix an hermitian metric on $\cO(D)$ in such a way that the equality is an isometry) and the Lemma holds for $M$, then the Lemma holds for $L$. Indeed, let $s\in H^0(Y,L^d)$ then , denoting by  $s_D\in H^0(Y, \cO(D)$ a section such that $\div(s_D)=D$, we have that $s\cdot   s_D^d\in H^0(Y,M^d)$. Consequently, if we denote by $\Vert s\Vert_{r_i}$ the norm $\sup_{z\in\Delta_{r_i}}\{\Vert s\Vert(z)\}$,  since the Lemma holds for $M$, we have
\begin{equation}
\Vert s\cdot s_D^d\Vert_{r_1}\leq C^d\Vert s\cdot s_D^d\Vert_{r_0}\leq C^d\Vert s\Vert_{r_0}\cdot \Vert s_D\Vert^d_{r_0}.
\end{equation}
A standard application of Jensen formula gives  the existence of a constant $C_0$ such that 
$\Vert s\Vert_{r_1}\cdot\Vert s_D\Vert^d_{r_1}\leq C_0^d\cdot\Vert s\cdot s_D^d\Vert_{r_1}$. The claim follows.  

Thus we may suppose that $L=\cO(nP)$ for $P\not\in U$ and $n$ sufficiently big. Consequently $Y$ is embedded in the projective space $\P^N$, the  line bundle $L$ is the restriction to $Y$ of the tautological line bundle $\cO(1)$ and that there is an hyperplane $H$ which do not intersect $U$. Let $V:=\P^N\setminus H$ and $Z=Y\cap V$. We are in position to apply Theorem \ref{sadullaev1} and conclude by choosing the suitable constants needed to compare the involved metrics.\end{proof}

From this we deduce:

\begin{lemma}\label{sadullaev2} Let $Y\subset X_{\sigma_0}$ be an algebraic curve (not necessarily defined over $\overline{K}$). Let $U\subset Y$ be a (non compact) simply connected open set whose complementary has non empty interior. Fix a biholomorphic isomorphism $\varphi_U:\Delta_1\to U$. Fix two positive integers $0<r_0<r_1<1$. Then we can find a constant $C_1$ such that, for every positive integer $d$ and every $s\in H^0(X_{\sigma_0}, \cL_{\sigma_0}^d)$ we have that
$$\varphi_U^\ast(s)\in B_{C_1}(r_0,r_1,\cL^d|_U).$$
\end{lemma}

The last lemma is of topological nature. We would like to thank T. Delzant who provided the main idea of the proof of it.

\begin{lemma}\label{coverings} Let $Z$ be a compact Riemann surface. Then we can find a finite set of coverings $\cU_j=\{U_{ij}\}$, $j=1,\dots, r$ with the following properties:

a) Each set $\cU_j:=\{ U_{ij}\}$ is finite.

b) Each $U_{ij}$ is non compact and simply connected. 

c) Each time we fix  (for every $i$ and $j$) a holomorphic isomorphism  $\varphi_{ij}:\Delta_{1}\to U_{ij}$, we can find a real number $0<r_0<1$ for which the following holds: denote by $W_{ij}$ the open set $\varphi_{ij}^{-1}(\Delta_{r_0})$. Then, for each $j$, $Z=\cup_iW_{ij}$ (this means that, for $j$ fixed, the open sets $W_{ij}$ still cover $Z$).

d) For every $z\in Z$ there is a $j_z\in\{ 1\dots, n\}$ such that $z\in \cap_{i}W_{ij_z}$.
\end{lemma}

\begin{proof} We can find a finite convering $\cU_0=\{ U_i\}$ with the following property:

-- Each $U_i$ is simply connected and non compact. Each $U_i$ contains a open set $V_i$ such that $\overline V_i\subset U_i$,  and  $Z=\cup_i V_i$. 

-- The intersection $\cap_i V_i$ is non empty.




Denote by $V:=\cap_i V_i$. Fix a point $z_0\in V$. 

For each $z\in Z$ we fix a diffeomorphism $f_z:Z\to Z$ such that $f(z)=z_0$. Hence $Z=\cup_{z\in Z} f_z^{-1}(V)$.  Since $Z$ is compact, we can find a finite set $z_1,\dots, z_n$ such that $Z=\cup_{h=1}^n f_{z_h}^{-1}(V)$. For each $h=1,\dots, n$, denote by  $U_{ih}$ (resp. $V_{ih}$) the open set $f_{z_h}^{-1}(U_i)$  (resp.  $f_{z_h}^{-1}(V_i)$) and by  $\cU_h$ the covering $\{ f_{z_h}^{-1}(U_i)\}$.

Since each $f_h$ is a diffeomorphism, all the $U_{ih}$ are simply connected and non compact. Since  $Z=\cup_{z\in Z} f_z^{-1}(V)$, for every $z\in Z$, there exists $h_z$ such that $z\in \cap_i V_{ih_z}$. 

For every $i$ and $h$ fix a holomorphic isomorphism $\varphi_{ih}:U_{ih}\to \Delta_1$. Since $\overline V_{ih}\subset U_{ih}$, for every $i$ and $h$, there is a real $0<r_{ih}<1$ such that $V_{ih}\subset \varphi_{ih}^{-1}(\Delta_{r_{ih}})$.

Let $r_0=\max\{r_{ih}\}$ and put $W_{ih}=\varphi_{ih}^{-1}(\Delta_{r_0})$. 

It is easy to verify that properties (c) and (d) are verified by $r_0$ and the collection of coverings $\cU_h$,  $h=1,\dots, n$. \end{proof}

We can now prove Theorem \ref{SNongenericcurves}.

\begin{proof} ({\it of Theorem \ref{SNongenericcurves})} We fix  a finite set of coverings $\cU_j$ of $Y$ which verify Lemma \ref{coverings}. 

For every positive integer $d$, positive real constants $B_i$ and  non zero section $s\in H^0(\cX, \cL_{\sigma_0}^d)$ denote by $V_Y(s, B_1, B_2)$ the area of the set $\{ z\in Y\;/\; \Vert s\Vert_{\sigma_0}\leq {{\Vert s\Vert_Y}\over{\Vert s\Vert^{B_1d^{n+1}}_+B_2^{d^{n+2}}}}\}.$

We now claim the following: 

For every sufficiently positive constant $C_0$ we can find constants $C_i>1$ depending only on $Y$, the arithmetic polarization  and the constants $B_i$ for which  the following holds:

For every positive integer $d$ and non zero section $s\in H^0(\cX, \cL_{\sigma_0}^d)$, we have
\begin{equation}\label{eq:area}V_Y(s,B_1,B_2)\leq {{C_0}\over
{\Vert s\Vert_+^{C_1d^{n}}C_2^{d^{n+1}}}}.\end{equation}

\

{\it Proof of the claim: } Fix a finite family of coverings as in Lemma \ref{coverings}. Let $r_0$ be the involved real number (Property (c)). Fix a real number $r_1$ such that $0<r_0<r_1<1$.   Let $s\in H^0(Y, \cL_{\sigma_0}^d)$.  Let $z_0\in Y$ such that $\sup_{z\in Y}\{\Vert s\Vert_{\sigma_0}\}=\Vert s\Vert_{\sigma_0}(z_0)$.  We can find one of the coverings $\cU_j$ with the property that $z_0$ belongs to all the open sets of the covering. We apply Lemma \ref{sadullaev2} and Lemma \ref{interpolation} to each of the open sets of the involved covering with $\epsilon =1/(\Vert s\Vert^{B_1d^{n+1}}_+B_2^{d^{n+2}})$ and the claim follows. 

An important issue which follows from the proof of the claim (and that will be used in the following) is that $lim_{B_i\to +\infty}C_j=+\infty$. By this we mean that, up to increase the $B_i$ if necessary, we may suppose that the $C_j$ are as big as we want. 

We are now in position to conclude the proof of the Theorem. The conclusion will be a direct application of Borel--Cantelli Lemma \ref{borelcantelli}.

Fix constants $B_i$' sufficiently big (how big will be fixed at the end of the proof). By the claim above we have that (for a suitable constant $C_3$):
\begin{eqnarray*}
&&\sum_{d=1}^\infty\sum_{s\in H^0(\cX;\cL^d)}V(s, B_1,B_2) \\
&&\leq \sum_{d=0}^\infty\sum_{s\in H^0(\cX;\cL^d)} {{C_0}\over
{\Vert s\Vert_+^{C_1d^{n}}C_2^{d^{n+1}}}} \;\;\;\;\;\;\;\text {by  formula \ref{eq:area}}\\
&& =\sum_{d=1}^\infty\sum_{N=1}^\infty\sum_{  {s\in H^0(\cX;\cL^d)}\atop
{N\leq \Vert s\Vert_+<N+1}}  {{C_0}\over
{\Vert s\Vert_+^{C_1d^{n}}C_2^{d^{n+1}}}}\\
&&\leq \sum_{d=1}^\infty\sum_{N=1}^\infty{{C_0C_3^{d^{n+1}}N^{d^n}}\over{N^{C_1d^n}C_2^{d^{n+1}}}}\;\; \;\;\;\;\;\;\;\;\;\;\;\;\;\;\text {by formula \ref{eq:counting1}}
\end{eqnarray*}

and the last series converges as soon as $C_2$ is sufficiently big and $C_1>1$. By the Borel Cantelli lemma \ref{borelcantelli} the conclusion follows. \end{proof}

\section{$S_a $ points are full if $a\geq N+1$}.

We suppose that we are in the same hypotheses of the previous section. 

Given a compact sub variety $Y$ of $X_\sigma$, as Theorem \ref{liouville4} shows,  the term $\log\Vert s\Vert_Y$ may be,  in general, quite difficult to control. It is possible that it is much smaller thet the term $\log^+\Vert s\Vert$. Never the less we will now show that the set $S_a(X_K)$ is a full set in $X_{\sigma_0}(\C)$.

\begin{theorem}\label{SNFULL} If $a\geq N+1$, the set $S_a(X_K)$ is full in $X_{\sigma_0}(\C)$.
\end{theorem}
\begin{proof} The proof is by induction on the dimension of the variety. If $\dim(X_K)$ is one, Theorem \ref{SNongenericcurves} apply directly and the conclusion follows in this case. Thus, by induction,  we may suppose  that $N=\dim(X_K)\geq 2$, that $a=N+1$, and that,  for every generically smooth divisor $D$ of $X_K$, the set $S_{N}(D_\sigma)$ is full in $D_\sigma(\C)$.

Fix a section $s_0\in H^0(\cX, \cL)$ and denote by $\cD\subset\cX$ the divisor $\div(s_0)$. By Bertini theorem we may suppose that its generic fiber $D_K$ is smooth.  Fix a point $z_0\in S_{N}(D_K)$.  The set of these points is full in $D_{\sigma_0}(\C)$  by inductive hypothesis. 

Let $Y\subset X_{\sigma_0}(\C)$ be a closed arithmetically generic  curve passing through $z$.

The proof will be a consequence of this lemma, the proof of which will be postponed to the end of the proof. 

\begin{lemma}\label{reductionlemma} There is a constant $A$,  depending only on $Y$ and $z_0$ for which the following holds: For every positive integer $d$ and non vanishing section $s\in H^0(\cX,\cL^d)$ we have that
$$\log\Vert s\Vert_Y\geq -Ad^N(\log^+\Vert s\Vert +d).$$
\end{lemma}

Let's show how the lemma implies the theorem. Theorem \ref{SNongenericcurves} and Lemma \ref{reductionlemma} imply that, if $z_0\in Y$ and $Y$ is arithmetically generic, then $S_{N+1}(X_K)\cap Y$ is full in $Y$. Choose a fibration by curves $h:X_K\to B$ for which the induced morphism $h_{D_K}:D_K\to B$ is generically finite. The image $h_{D_K}(S_{N}(D_K))$ will be full in $B$ and by Proposition \ref{fibrations and Sa} contained in $S_{N}(B_K)$.  Consequently for every $z_0\in S_{N}(D_K)$, if we denote by $Y_{z_0}$ the fibre $h^{-1}(h(z_0))$, we have that $S_{N+1}(X_K)\cap Y_{z_0}$ is full in $Y_{z_0}$. Fubini Theorem implies the conclusion of the Theorem. \end{proof}

\begin{proof} {(\it of Lemma \ref{reductionlemma})} Let $s\in H^0(\cX,\cL^d)$ a non vanishing section. 
If $s(z_0)\neq 0$ then, since $z_0\in S_{N}(D_K)$ and the restriction of $s$ to $D_K$ is non zero, we have that 
\begin{equation}
\log\Vert s\Vert(z_0)\geq -Ad^{N}(\log^+\Vert s\Vert_{D_K} +d)\geq -Ad^{N}(\log^+\Vert s\Vert+d)
\end{equation}
for a suitable constant $A$ independent on $s$ and $d$. Thus, in this case, the conclusion of the Lemma holds for $s$ (because $z_0\in Y$).

Since $Y$ is arithmetically generic, the restriction of $D_{\sigma_0}$ to $Y$ is not identically zero. Thus, the restriction of $\div(s_0)$ to $Y$ is an effective divisor vanishing on $z_0$. Denote by $b$ the multiplicity  at $z_0$ of $\div(s_0)|_Y$ and by $J^b(s_0)$  the $b$--th jet at $z_0$ of the section $s_0|_Y$.

Suppose now that $s(z_0)=0$. This implies that we can find a constant $\alpha\leq d$ such that $\div(s)=\alpha D+D_1$ with $D_1$ not vanishing on $z_0$. Indeed, since $z_0$ is arithmetically generic on $D$, every non zero global section of $\cL^d|_D$ will not vanish on $z_0$. Consequently the order of vanishing at $z_0$ of the restriction of $s$ to $Y$ is $\alpha b$. Denote by $J^{\alpha b}(s)$ the $\alpha b$--th jet of $s|_Y$ in $z$. 

Denote by $s_1\in H^0(\cX,\cL^{d-\alpha})$ the global section $s/s_0$. Since $s_1$ do not vanish at $z$, as before, we have that

\begin{equation}\label{eq:estimatesfors} 
\log\Vert s_1\Vert_{\sigma_0}(z_0)\geq -Ad^{N}(\log^+\Vert s_1\Vert +d)
\end{equation}
for a suitable constant $A$ independent on $s_1$.
A local computation give the existence of a constant $C_1$ such that
\begin{equation}\label{eq:jetnorms} \log\Vert J^{\alpha b}(s)\Vert (z)\geq \log\Vert J^b( s_0)\Vert(z_0)+\log\Vert s_1\Vert_{\sigma_0}(z)+C_1d.
\end{equation}

Thus, from \ref{eq:estimatesfors} and \ref{eq:jetnorms} we conclude that, we can find a constant $A_2$ independent on the section $s$, such that
$$\log\Vert J^{\alpha b}(s)\Vert (z_0)\geq -A_2d^{N}(\log^+\Vert s_1\Vert +d).$$
which, together with \ref{eq:supestimate}, gives the existence of a constant $A_3$ (independent on $s$) such that
\begin{equation}\label{eq:jetestimate1} \log\Vert J^{\alpha b}(s)\Vert (z_0)\geq -A_3d^{n}(\log^+\Vert s\Vert +d).
\end{equation}
Let $\Delta_1\subset Y$ be a holomorphic disk centered in $z$ and with coordinate $\zeta$. By Nevanlinna first main theorem applied  to the inclusion $\Delta_1\hookrightarrow X_{\sigma_0}(\C)$ we obtain that there exists a constant $C$ such that
$$d\cdot C +\int_{|\zeta|=1}\log\Vert s\Vert_{\sigma_0} d\theta\geq \log\Vert J^{ab}(s)\Vert (z).$$
Thus there exists  a point $\zeta_0\in Y$ such that 
$$\log\Vert s\Vert (\zeta_0)\geq  -A_3d^{N}(\log^+\Vert s\Vert +d).$$ \end{proof}

\

We would like to propose the following conjecture which, if true, may imply further interesting speculations:

\begin{conjecture} Let $X$ be a smooth projective variety defined over a number field of dimension $n$ defined over a number field. Let $a\geq N$ be a real number. Then the set of points of type $S_a(X)$ is full in $X(\C)$. \end{conjecture}

Even if not stated in this language, the conjecture above is proved, for instance in \cite{BUGEAUDBOOK} when $X=\P^1$.

\section{Some applications to rational points on analytic disks}

\

Even if this paper is more about transcendental points on algebraic varieties, we will see how this theory can be an interesting tool to study rational points. 

As before, we suppose that $X$ is a smooth projective variety of  dimension $N>1$ defined over a number field $K$. We fix an arithmetic polarization $(\cX,\cL)$ of $X$. 

Let $\sigma$ be an infinite place of $K$  and 
$f:\Delta_1\to X_\sigma$ be an analytic map ($\Delta_1$ being the unit disk in $\C$). We are interested on studying the set $f^{-1}(X(K))$. 

For every positive real numbers $T$, and $r<1$ we introduce the set
\begin{equation}
S_r(f,T):=\{ z\in \Delta\; /\; |z|<r\; ; \; f(z)\in X(K)\; {\rm and}\; h_\cL(f(z))\leq T\}
\end{equation}
and we denote by $C_r(f,T)$ its cardinality. 

In this section we are interested in estimating $C_r(f,T)$ when $T$ goes to infinity. 

In the classical paper \cite{BOMBIERIPILA}, authors show that, for every positive $\epsilon$ we have an estimate of the form $C_r,f,T)\ll \exp(\epsilon T)$ and examples by Pila \cite{PILA} and Surroca \cite{surroca}(independently) show that, in general, one cannot hope better then this. 

If some conditions are given on $f$ then the bound by Bombieri and Pila can be drastically improved. A huge literature on this topic is available, cf. for instance \cite{BOXALL}, \cite{COMTEYOMDIN}, \cite{GMW}, \cite{MASSER}.

In this section we will show how points of type $S_a$ may be used to control the growth of rational points in the image of the analytic map $f$. We think that this circle of ideas may be expanded and we hope that we will improve this in a future paper.

Problems similar to this have been studied by many authors: for instance Lang treat the case of maps analytic maps of $\C$ in his classical book \cite{LANG}. Two  papers which inspire this one are the already quoted \cite{MASSER} and \cite{BOXALL}.

The main theorem of this section is

\begin{theorem}\label{counting2} Suppose that there is $z_0\in\Delta_1$ such that $f(z_0)\in S_a(X)$.  Then, for every $\epsilon>0$ and $\gamma>{{1}\over{N}}$ there exists a constant $A=A(\cX,\cL, r,f,\epsilon, \alpha, \gamma)$ such that, if $T\geq A$, we have
\begin{equation}
C_f(r,T)\leq \epsilon T^{1+\gamma(a+1)}.
\end{equation}
\end{theorem}

Observe that the hypothesis of the theorem imply that the image of $f$ is Zariski dense (over the algebraic closure of $K$). 

The main tool of the proof is the following Lemma, of independent interest, the proof of which is inspired by the proof of Proposition 2 of \cite{MASSER}.

\begin{lemma}\label{vanishingsmallsections} Let $f:\Delta_1\to X_\sigma$ be an analytic map whose image is Zariski dense. Fix $1>\epsilon_0>0$ and $\gamma>{{N}\over{N-1}}$. With the notations as above, there is a constant $A_0=A_0(\cX.\cL,f,r,\epsilon_0, \gamma)$ for which the following holds: for every $T\geq A_0$ there exists a non zero global section $s\in H^0(\cX, \cL^d)$ such that:

-- $d\leq \epsilon_0 T^{{{\gamma}\over{N}}}$;

-- $\log\Vert s\Vert \leq\epsilon_0 T^{1+{{\gamma}\over{N}}}$;

-- For every $z\in S_r(f,T)$ we have $s(f(z))=0$.
\end{lemma}

\begin{proof} Before we start the proof, we recall the following standard fact: if we denote by $h^0(X,L^d)$ the dimension of $H^0(\cX, \cL^d)$, then,  can find a positive constant $B_1$ such that, for every $\epsilon_1>0$ and $d$ sufficiently big, we have 
\begin{equation}
B_1(1-\epsilon_1)d^n\leq h^0(X,L^d)\leq B_1(1+\epsilon_1) d^n.
\end{equation}
We  fix such a $\epsilon_1<{{1}\over{9}}$.  We also fix $\epsilon_2<(\epsilon_0/10)^N$.

We also recall that, by Zhang's theorem \cite{Zhang}, we may suppose that $H^0(\cX,\cL^d)$ is generated by sections of norm less or equal then one. 

We may suppose that, $A_0$ is so big that  for  $T\geq A_0$ we have that $\epsilon_2(1-\epsilon_1)B_1T^\gamma>1$ and $\epsilon_2^{1/N}(4^{1/N}-3^{1/N})T^{\gamma/N}>1$.

Let $A(T)$ be a positive integer such that $\epsilon_2(1-\epsilon_1)B_1T^\gamma\leq A(T)\leq 2\epsilon_2(1-\epsilon_1)B_1T^\gamma$. Choose a subset $H(T)$  of $S_r(f,T)$ of cardinality $A(T)$. 

For every positive integer $d$, denote by $V(T,d)$ the hermitian $O_K$ module $\oplus_{z\in H(T)}\cL^d|_{f(z)}$. The rank of $V(T,d)$ is $A(T,\epsilon)$ and $\widehat{\mu}_{\max}(V(T,d))\leq dT$ (the definition of $\widehat{\mu}_{\max}(\cdot)$ is given in section \ref{notations etc}).

We have a natural restriction map
\begin{equation}
\delta_T: H^0(\cX,\cL^d)\longrightarrow V(T,d).
\end{equation}

By Gromov theorem \ref{gromov}, if we put on $H^0(\cX, \cL^d)$ the $L_2$ hermitian structure and on $V(T,d)$ the direct sum hermitian structure, the logarithm of the norm of $\delta$ is bounded by $C_0d$ for a suitable constant $C_0$. 

We choose $d$ such that $3\epsilon_2 T^{\gamma}\leq d^N\leq 4\epsilon_2T^{\gamma}$. 

Denoty by $K(T)$ the kernel of $\delta_T$ and by $k(T)$ its rank. With our choices we have that
\begin{equation}
{{h^0(\cX;\cL^d)}\over{k(T)}}\leq {{B_1(1+\epsilon_1)4\epsilon_2 T^{\gamma}}\over{B_1(1-\epsilon_1)3\epsilon_2 T^{\gamma}-2\epsilon_2(1-\epsilon_1)B_1T^\gamma}}={{4(1+\epsilon_1)}\over{(1-\epsilon_1)}}<5.
\end{equation}
We may apply Siegel Lemma \ref{siegel} and we find that, for $T$ sufficiently big, there exists a non vanishing section $s\in H^0(\cX,\cL^d)$ with $d\leq (4\epsilon_2)^{1/N}T^{\gamma/n}$ such that:

-- $\log\Vert s\Vert \leq 10 \epsilon_2^{1/N}T^{1+{{\gamma}\over{N}}}$;

-- for every $z\in H(T)$ we have $s(f(z))=0$.

We will now show that, under the hypotheses above, the section $s$ we just constructed vanishes on every element of $S_r(f,T)$.

Let $z_0$ be an element of $S_r(f,T)$ which is not in $H(T)$. 

Chose $r_1>r$ such that, for every $z\in\Delta_r$ there is an automorphism $\varphi_z$ of $\Delta$ which sends $z$ in $0$ and  such that $\varphi(\Delta_r)\subset \Delta_{r_1}$. Thus  we may suppose that $z_0=0$. The reader will check, that by the compactness of $\Delta_r$, the constants which will appear in the estimates below, may be chosen independently on $z_0$ and depending only on $r$. 

Suppose that $s(f(0))\neq 0$.

By the standard Liouville inequality, Corollary \ref{Liouville and heights}, we have that
\begin{equation}
\log\Vert s\Vert (f(0))\geq -dT-([K:\Q]-1)\log\Vert s\Vert\geq -A_2T^{1+{{\gamma}\over{N}}}
\end{equation}
for a suitable constant $A_2$ independent on $z_0$. 

We  apply Nevanlinna first main theorem and we obtain:
\begin{equation}
dT_f(r_1)+\int_{|z|=r_1}\log \Vert s\Vert (r_1e^{i\theta})d\theta\geq \sum_{z\in H(T)}\log{{r_1}\over{|z|}}+\log\Vert s\Vert (f(0)).
\end{equation}
This, together with the Liouville inequality above and the properties of $s$ gives the existence of constants $A_4$ and $A_5$ depending on $r_0$ but independent on $z_0$ and on $T$, such that
\begin{equation}
A_4 T^{1+{{\gamma}\over{N}}}\geq A_5T^\gamma
\end{equation}
But this cannot hold because of our choice of $\gamma$, as soon as $T$ is sufficiently big. Consequently $s$ should vanish on $f(z_0)$ and the conclusion of the Lemma follows. \end{proof}

We can now prove Theorem \ref{counting2}: 

\begin{proof} As in the proof of \ref{vanishingsmallsections}, we may suppose that $z_0=0$.

Consider the section $s\in H^0(\cX, \cL^d)$ we constructed in Lemma \ref{vanishingsmallsections} with the involved $\epsilon_0$ sufficiently small.  

Since $f(0)\in S_a(X_K)$, we can choose $\epsilon_0$ in such a way that 
\begin{equation}
\log\Vert s\Vert (f(0))\geq -Ad^{a}(\log\Vert s\Vert+d)\geq -\epsilon_0 T^{a\gamma+1+\gamma}.
\end{equation}

By Nevanlinna First Main Theorem applied to $s$, we thus can find constants  $C_i$ independent on $T$ such that
\begin{equation}
dC_0+\int_{|z|=r_1}\log\Vert s\Vert(r_1e^{i\theta})d\theta\geq C_r(f,T)C_1-\epsilon_0 T^{\gamma(a+1)+1}
\end{equation}
The conclusion follows from our choice of $d$, the bound on $\log\Vert s\Vert$ and a suitable choice of $\epsilon_0$.\end{proof}

Theorem  \ref{counting2} tells us that, as soon as the image of an analytic map from a disk to $X$ contains an arithmetically generic point with some good transcendental properties, then it contains "few" rational points. 

We now show that a minor modification of the proof tells us that a similar result is obtained if the image intersect an effective ample divisor in a "good" point:

\begin{theorem}\label{counting3} Let $f:\Delta_1\to X(\C)$ as before. Let $a\geq N-1$ be a real number. Let $s_0\in H^0(\cL, \cL)$ be an irreducible smooth divisor. Suppose that  there is $p\in f(\Delta_r)\cap div(s_0)$ which is of type $S_a(\div(s_0))$. Then, for every $\epsilon>0$ and $\gamma\geq {{1}\over{N}}$ there exists a constant $A=A(\cX,\cL, r,f,\epsilon, \a)$ such that, if $T\geq A$ then we have
\begin{equation}
C_f(r,T)\leq \epsilon T^{\gamma(a+1)+1}.
\end{equation}
\end{theorem}

\begin{proof} Since the intersection of $f(\Delta_r)$ and $\div(s_0)$ is finite, we can consider just points which are not in $\div(s_0)$. Let $s\in H^0(\cX,\cL^d)$ be the section constructed via Lemma \ref{vanishingsmallsections}. 

Write $div(s)=\alpha div (s_0)+div(s_1)$.  By construction, the restriction of $s_1$ to $div(s_0)$ do not vanish identically. By inequality \ref{eq:supestimate} we  can find positive constant $C_1$ , $\epsilon_0$ and $\epsilon_1$ such that:
\begin{equation}
\epsilon_0 T^{{{N}\over{N-1}}+\epsilon_1}\geq \log\Vert s\Vert\geq \alpha \log\Vert s_0\Vert +(d-\alpha)C_1+\log\Vert s_1\Vert .
\end{equation}
Thus, by our choice of $d$, increasing $T$ and modifying $\epsilon_0$ if necessary, we have
\begin{equation}
 \epsilon_0 T^{{{N}\over{N-1}}+\epsilon_1}\geq\log\Vert s_1\Vert.\
\end{equation}
Since $s_1$ do not vanish on $div(s_0)$ it do not vanish on $p$ (because, in particular $p$ is arithmetically generic in $div(s_0)$). Consequently. since $p$ is of type $S^a(div(s_0))$, 
\begin{equation}
\log\Vert s_1\Vert(p)\geq -Ad^a(\log\Vert^+ s_1\Vert +d).
\end{equation}
The conclusion follows as in the proof of Theorem \ref{counting2}. \end{proof}

\end{document}